\documentclass[11pt]{article}
\usepackage{color}
\usepackage[margin=1in]{geometry}
\usepackage{amssymb}
\usepackage{amsmath}
\usepackage{amsthm}
\usepackage{graphicx}
\usepackage{cases}
\usepackage{algorithm,algorithmic}

\newtheorem{theorem}{Theorem}

\newtheorem{definition}{Definition}
\newtheorem{example}{Example}
\newtheorem{property}{Property}
\newtheorem{corollary}{Corollary}

\newtheorem{assumption}{Assumption}
\newcommand{\R}{\mathbb{R}}
\newcommand{\T}{\mathcal{T}}
\newcommand{\norm}[1]{\left\lVert#1\right\rVert}
\newcommand{\itp}{\hat{\otimes}_\varepsilon}

\title{A hybrid iterative method based on MIONet for PDEs: Theory and numerical examples\footnote{The research is supported by NSFC project 12288101.}}
\author{Jun Hu\footnote{School of Mathematical Sciences, Peking University, Beijing 100871, China (hujun@math.pku.edu.cn).}\ \ and\ Pengzhan Jin\footnote{School of Mathematical Sciences, Peking University, Beijing 100871, China (jpz@pku.edu.cn).}
}
\date{}

\begin{document}
	
	\maketitle
\begin{abstract}
We propose a hybrid iterative method based on MIONet for PDEs, which combines the traditional numerical iterative solver and the recent powerful machine learning method of neural operator, and further systematically analyze its theoretical properties, including the convergence condition, the spectral behavior, as well as the convergence rate, in terms of the errors of the discretization and the model inference. We show the theoretical results for the frequently-used smoothers, i.e. Richardson (damped Jacobi) and Gauss-Seidel. We give an upper bound of the convergence rate of the hybrid method w.r.t. the model correction period, which indicates a minimum point to make the hybrid iteration converge fastest. Several numerical examples including the hybrid Richardson (Gauss-Seidel) iteration for the 1-d (2-d) Poisson equation are presented to verify our theoretical results, and also reflect an excellent acceleration effect. As a meshless acceleration method, it is provided with enormous potentials for practice applications.
\end{abstract}

\section{Introduction}
Partial differential equations (PDEs) play a crucial role in describing physical reality and modeling engineering problems. As the vast majority of PDEs have no computable analytical solutions, seeking numerical solutions is an important issue. There have been many classical numerical methods developed for solving PDEs, such as the finite element method (FEM) \cite{brenner2008mathematical}, the finite difference method (FDM) \cite{smith1985numerical}, the finite volume method \cite{eymard2000finite}, and the spectral method \cite{shen2011spectral}. With these methods, linear PDEs are usually discretized as linear algebraic systems, and subsequently be solved by direct methods (e.g. Gaussian elimination), or iterative methods such as Richardson, (damped) Jacobi, Gauss-Seidel, SOR, SSOR \cite{trefethen2022numerical}. When the scale of the problem is very large, iterative methods show significant advantages like low memory requirements. However, a large scale system results in a long solving time, which is unacceptable. To accelerate the iteration, multilevel iterative methods are developed \cite{hackbusch2013multi,trottenberg2000multigrid,xu1989theory}, a multigrid solver will greatly reduce the steps of iteration and lead to fast convergence. The shortcoming of the multigrid method is that it is mesh-depended, and requires users to generate several levels of meshes, which may be a very heavy task. In order to overcome this drawback, the algebraic multigrid (AMG) methods \cite{brandt1986algebraic,mandel1988algebraic,ruge1987algebraic,xu2017algebraic} are proposed, which deal with the algebraic system directly without geometric meshes. To further promote the development of fast and low usage cost iterative methods, in this work we propose a hybrid iterative method by introducing neural network-based tools.

The rise of neural networks (NNs) brings a huge impact on searching numerical solutions of PDEs. It has been shown that NNs have powerful expressivity \cite{e2019barron,siegel2022high}, thus NNs are considered as universal approximators for the solutions of PDEs. Mainstream neural networks-based solving methods for PDEs are optimizing the parameters of such NNs to minimize the corresponding loss functions constructed by strong or variational forms of PDEs, e.g. the PINNs \cite{karniadakis2021physics,lu2021physics,raissi2019physics} , the deep Ritz method \cite{yu2018deep}, and the deep
Galerkin method \cite{sirignano2018dgm}. These methods show the possibility for solving high-dimensional problems, which is indeed a challenge for traditional methods due to the curse of dimensionality. As meshless methods, they are easily implemented and friendly for coders. The drawbacks of such methods are also obvious. On the one hand, these methods are usually hard to obtain solutions as accurate as the traditional methods like FEM, on the other hand, one needs to restart a training process once some parameters of the given PDE problem (e.g. the right-hand side term of Poisson equation) are changed, which is costly to deal with parametric PDE problems. In recent years, a new research field called neural operator, is emerging to achieve the fast solving for PDEs based on neural networks. The pioneer works are the DeepONet \cite{lu2019deeponet,lu2021learning} and the FNO \cite{li2020fourier,li2020neural}, they directly learn the solution operators of PDEs, for example, the mapping from the right-hand side of Poisson equation to its solution. These methods present a new landscape for understanding physical reality, one may study the causality via a data-driven manner besides the PDE model. Recently numerous neural operators are proposed to achieve higher performance \cite{gupta2021multiwavelet,he2023mgno,jiang2023fourier,jin2022mionet,rahman2022u,raonic2023convolutional,wu2023solving}, as this field is developing rapidly. There are also theoretical works for the above neural operators including error analysis \cite{kovachki2021universal,lanthaler2022error,marcati2023exponential}.

The most attractive neural operator for this work is the multiple-input operator network (MIONet) \cite{jin2022mionet}. MIONet firstly faces up to the problems with multiple inputs, and generalizes the architecture and the theory of DeepONet by the tool of tensor decomposition. Such architecture also induces the tensor-based models for solving eigenvalue problems \cite{hu2023experimental, wang2022tensor}. An important characteristic of DeepONet/MIONet series architectures is that they encode the output functions as neural networks, rather than discrete vectors. This feature leads to efficient interpolation for the predicted solutions, and also allows the computing of derivatives of output functions during training, which prompts the work of physics-informed DeepONets \cite{wang2021learning} that learns the solution operator without any input-output pairs of data, by training a PINN-style loss. Another interesting work based on DeepONet is the HINTS \cite{zhang2022hybrid}, which in fact inspires this work. Almost all of the works of neural operators use a pure machine learning-style to solve the PDE problems, and they break away from traditional PDE numerical methods. The neural operators are  developed for fast solving, but usually have limited prediction accuracy, while the traditional numerical methods can provide solutions as accurate as needed at the cost of long solving time. Surprisingly, HINTS combines the two methods, i.e. the traditional numerical solver and the neural operator. It has been shown a spectral bias/frequency principle of the training behavior of NNs in \cite{rahaman2019spectral,xu2019frequency}, i.e. the NNs often fit functions from low to high frequency during the training. Based on this feature, \cite{huang2020int} employs the deep learning solutions (e.g. PINNs) to initialize the iterative methods for PDEs. HINTS further utilizes the spectral bias, where a pre-trained DeepONet is periodically used to eliminate the low-frequency components of the error during the iteration process, thus achieve acceleration since the smoothers (e.g. Gauss-Seidel) are difficult to deal with the low-frequency errors. This idea is similar to the mulltigrid method. A significant difference is that HINTS does not require the users to generate meshes, and the cost of this method has been transferred to the developers who are responsible for the pre-trained DeepONets.

\textbf{Contributions.} As HINTS studies the hybrid method empirically, the lack of theoretical guidance leads to the imperfection of this method. In this work, we propose a hybrid iterative method based on MIONet. We theoretically show that MIONet is a more reasonable framework to complete the hybrid iterative method. An ideal operator regression model for the hybrid method is expected to satisfy the following three key points:
\begin{itemize}
\item[1.] \textbf{Supporting multiple inputs.} Taking the Poisson equation
\begin{equation}\label{eq:intro_poisson}
	\begin{cases}
		-\nabla\cdot(k\nabla u)=f &\quad \mbox{in} \; \Omega,
		\\
		u = 0 &\quad \mbox{on} \; \partial\Omega,
	\end{cases}	
\end{equation}
as an example, its solution operator $u=\mathcal{G}(k,f)$ has two inputs, hence the chosen operator regression model should efficiently deal with such multiple-input case.
\item[2.] \textbf{Separating linearity and nonlinearity.} Since the solution operator $\mathcal{G}(k,f)$ of (\ref{eq:intro_poisson}) is linear w.r.t. $f$ but nonlinear w.r.t. $k$, the operator regression model $\mathcal{M}(k,f)$ is also expected to be linear w.r.t. $f$ and nonlinear w.r.t. $k$. The linearity on $f$ will guarantee the convergence of the hybrid iterative method which will be discussed later.
\item[3.] \textbf{Supporting efficient interpolation.} After obtaining the predicted solution by operator regression model $\mathcal{M}$, we have to compute the values of $\mathcal{M}(k,f)(\cdot)$ at the grid points which are independent of the model $\mathcal{M}$, so that the model $\mathcal{M}$ should support efficient interpolation.
\end{itemize}
Fortunately, we find that MIONet satisfies all the three points perfectly, while other models do not (note that MIONet is a high-level framework, and one may design a targeted architecture for each branch net on a specific problem, rather than directly use FNNs.). With the proposed MIONet-based hybrid iterative method, we systematically analyze its theoretical properties, including the convergence condition, the spectral behavior, as well as the convergence rate, in terms of the errors of the discretization and the model inference. We show the theoretical results for the frequently-used smoothers, i.e. Richardson (damped Jacobi) and Gauss-Seidel. We give an upper bound of the convergence rate of the hybrid method w.r.t. the model correction period, which indicates a minimum point to make the hybrid iteration converge fastest. Several numerical examples including the Richardson (Gauss-Seidel) iteration for the 1-d (2-d) Poisson equation are shown to verify our theoretical results.

This paper is organized as follows. We first briefly introduce the operator regression model MIONet as the necessary tool in Section \ref{sec:mionet}. In Section \ref{sec:him}, we present the algorithm of the proposed hybrid iterative method in detail. We then comprehensively analyze the theoretical properties of this method in Section \ref{sec:theory}. The numerical experiments are shown in Section \ref{sec:experiments}. Finally, Section \ref{sec:conclusions} summarizes this work.
	
\section{Operator regression with MIONet}\label{sec:mionet}
Let us simply introduce the operator regression model MIONet, which is a necessary tool for the later discussion of the hybrid iterative method. All the content of this section can be found in the MIONet paper \cite{jin2022mionet}, and the readers may refer to that paper for more details if interested. The readers familiar with MIONet may skip this section.

Let $X_1,X_2,...,X_n$ and $Y$ be $n+1$ Banach spaces, and $K_i \subset X_i$ ($i=1,...,n$) is a compact set. Now we aim to learn a continuous operator
\begin{equation*}
    \mathcal{G}: K_1\times \cdots\times K_n \to Y,\quad (v_1, ..., v_n) \mapsto u,
\end{equation*}
where $v_i\in K_i$ and $u=\mathcal{G}(v_1,...,v_n)$. Such $\mathcal{G}$ form the space $C(K_1\times \cdots\times K_n,Y)$ which is studied below. Note that $Y$ could be any Banach space. We usually consider $Y$ as a function space, such as $Y=C(K_0),L^2(K_0),H^1(K_0)$ for a compact domain $K_0$.

To deal with the inputs from infinite-dimensional Banach spaces, it is necessary to firstly project them onto finite-dimensional spaces. Here we utilize the tools of Schauder basis and canonical projections.

\begin{definition}[Schauder basis]
Let $X$ be an infinite-dimensional normed linear space. A sequence $\{e_i\}_{i=1}^{\infty}$ in $X$ is called a Schauder basis of $X$, if for every $x\in X$ there is a unique sequence of scalars $\{a_i\}_{i=1}^{\infty}$, called the coordinates of $x$, such that $$x=\sum_{i=1}^{\infty}a_i e_i.$$
\end{definition}
We show two useful examples of Schauder basis as follows.
\begin{example} \label{exa:faber}
\textbf{Faber-Schauder basis of $C[0,1]$.} Given distinct points $\{t_i\}_{i=1}^{\infty}$ which is a dense subset in $[0,1]$ with $t_1=0$, $t_2=1$. Let $e_1(t)=1$, $e_2(t)=t$, and $e_{k+1}$ is chosen as an element, such that ${e_1,...,e_k,e_{k+1}}$ is a basis of the $(k+1)$-dimensional space which consists of all the piecewise linear functions with grid points $\{t_i\}_{i=1}^{k+1}$.
\end{example}
\begin{example}
\textbf{Fourier basis of $L^2[0,1]$.} Any orthogonal basis in a separable Hilbert space is a Schauder basis.
\end{example}

We denote the coordinate functional of $e_i$ by $e_i^*$, and thus
\begin{equation*}
    x=\sum_{i=1}^{\infty}e_i^*(x)e_i,\quad \forall x\in X.
\end{equation*}
Then for a positive integer $n$, the canonical projection $P_n$ is defined as
\begin{equation*}
    P_n(x)=P_n \left(\sum_{i=1}^{\infty}e_i^*(x)e_i\right) = \sum_{i=1}^{n}e_i^*(x)e_i.
\end{equation*}

We have the following property for $P_n$, according to which, we can represent points in an infinite-dimensional Banach space by finite coordinates within a sufficiently small projection error.

\begin{property}[Canonical projection] \label{pro:projection}
Assume that $K$ is a compact set in a Banach space $X$ equipped with a Schauder basis and corresponding canonical projections $P_n$, then we have
\begin{equation*}
    \lim_{n\to\infty}\sup_{x\in K}\norm{x-P_n(x)}=0.\\
\end{equation*}
\end{property}

For convenience, we decompose the $P_n$ as
\begin{equation*}
    P_n=\psi_n\circ\varphi_n,
\end{equation*}
where $\varphi_n:X\to\R^n$ and $\psi_n:\R^n\to X$ are defined as
\begin{equation*}
    \varphi_n(x) = \left(e_1^*(x),...,e_n^*(x)\right)^T,\quad\psi_n(\alpha_1,...,\alpha_n)=\sum_{i=1}^{n}\alpha_ie_i.
\end{equation*}
The $\varphi_n(x)$ are essentially the truncated coordinates for $x$. Moreover, sometimes we can further replace $\{e_1,...,e_n\}$ with an equivalent basis for the decomposition of $P_n$, i.e.,
\begin{equation*}
    \hat{\varphi}_n(x)=Q(e_1^*(x),...,e_n^*(x))^T,\quad\hat{\psi}_n(\alpha_1,...,\alpha_n)=(e_1,...,e_n)Q^{-1}(\alpha_1,...,\alpha_n)^T,
\end{equation*}
with a nonsingular matrix $Q\in\R^{n\times n}$. For example, when applying the Faber-Schauder basis (Example \ref{exa:faber}), instead of using the coordinates based on the sequence $\{e_i\}_{i=1}^\infty$, we adopt the function values evaluated at certain grid points as the coordinates, which is the same as the linear element basis in the finite element method.

By considering the canonical projection property and the injective tensor product decomposition
\begin{equation}\label{eq:key}
    C(K_1\times K_2\times\cdots\times K_n,Y)\cong C(K_1)\itp C(K_2)\itp\cdots\itp C(K_n)\itp Y,
\end{equation}
we can easily obtain the following approximation result.
\begin{theorem} \label{thm:approximation}
Suppose that $X_1,...,X_n,Y$ are Banach spaces, $K_i\subset X_i$ are compact sets, and $X_i$ have a Schauder basis with canonical projections $P_q^i=\psi_q^i\circ\varphi_q^i$ (truncate the previous q terms). Assume that $\mathcal{G}:K_1\times\cdots\times K_n\to Y$ is a continuous operator, then for any $\epsilon>0$, there exist positive integers $p,q_i$, continuous functions $g_{j}^i\in C(\R^{q_{i}})$, $u_j\in Y$, such that
\begin{equation}\label{eq:elementwise}
\sup_{v_i\in K_i}\norm{\mathcal{G}(v_1,...,v_n)-\sum_{j=1}^{p}g_j^1(\varphi_{q_1}^1(v_1))\cdots g_j^n(\varphi_{q_n}^n(v_n))\cdot u_j}_{Y}<\epsilon.
\end{equation}
In addition, if $\mathcal{G}$ is linear with respect to $v_i$, then linear $g_j^i$ is sufficient.
\end{theorem}

The MIONet is immediately constructed based on this theorem, where we just replace the $g_j^i$ with neural networks. Assume that now $Y=C(K_0)$ ($L^2(K_0)$, $H^1(K_0)$ etc. are also available) for a compact domain $K_0$.

\textbf{MIONet.} We construct the MIONet according to Eq.~(\ref{eq:elementwise}). Specifically, $\mathbf{g}_i=(g_1^i,...,g_p^i)^T\in C(\R^{q_i},\R^p)$ and $\mathbf{f}=(u_1,...,u_p)^T\in C(K_0,\R^p)$ are approximated by neural networks $\tilde{\mathbf{g}}_i$ (called branch net $i$) and $\tilde{\mathbf{f}}$ (called trunk net). Then the network (Figure \ref{fig:architecture_MIONet}) is
\begin{equation} \label{eq:MIONet_low}
    \tilde{\mathcal{G}}(v_1,...,v_n)(y) = \mathcal{S}\left( \underbrace{\tilde{\mathbf{g}}_1(\varphi_{q_1}^1(v_1))}_{\text{branch}_1} \odot\cdots\odot \underbrace{\tilde{\mathbf{g}}_n(\varphi_{q_n}^n(v_n))}_{\text{branch}_n} \odot \underbrace{\tilde{\mathbf{f}}(y)}_{\text{trunk}} \right) + b,
\end{equation}
where $\varphi_{q_i}^i$ truncates the previous $q_i$ coordinates w.r.t. the corresponding Schauder basis as defined in Eq. (\ref{eq:elementwise}), $\odot$ is the Hadamard product (element-wise product), e.g.
\begin{equation}
(x_1,...,x_p)\odot(y_1,...,y_p)=(x_1y_1,...,x_py_p),
\end{equation}
and $\mathcal{S}$ is the summation of all the components of a vector, $b\in\R$ is a trainable bias.

The universal approximation theorem for MIONet can be easily obtained via Theorem \ref{thm:approximation} and the universal approximation theorems for the applied neural networks (i.e. $\tilde{\mathbf{g}}_i$ and $\tilde{\mathbf{f}}$) such as FNNs.

\begin{theorem}[Universal approximation theorem for MIONet] \label{thm:UAT_MIONet}
Under the setting of Theorem \ref{thm:approximation} with $Y=C(K_0),L^2(K_0),H^1(K_0)$ etc., for any $\epsilon>0$, there exists a MIONet $\tilde{\mathcal{G}}$, such that $$\norm{\mathcal{G}-\tilde{\mathcal{G}}}_{C(K_1\times\cdots\times K_n,Y)}<\epsilon.$$
In addition, the branch nets can be chosen as linear maps (i.e. one-layer FNN without activation and bias) if $\mathcal{G}$ is linear with respect to the corresponding inputs.  
\end{theorem}

Up to now, we have already shown the architecture and properties of MIONet, which will be applied to the iterative methods later.

\begin{figure}[htbp]
    \centering
    \includegraphics[width=0.9\textwidth]{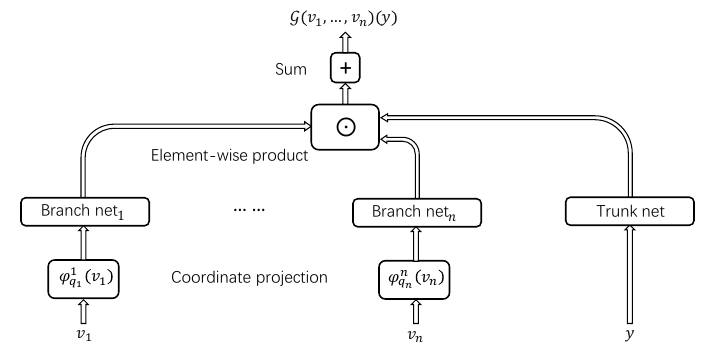}
    \caption{\textbf{An illustration of the architecture of MIONet for general operators.} All the branch nets and the trunk net have the same output dimension, whose outputs are merged together via the Hadamard product and then a summation.}
    \label{fig:architecture_MIONet}
\end{figure}

\section{A hybrid iterative method}\label{sec:him}
We present a new MIONet-based hybrid iterative method for general linear equations:
\begin{equation}
	\begin{cases}
		\mathcal{L}_{\alpha}u=f &\quad \mbox{in} \; \Omega,
		\\
		u = g &\quad \mbox{on} \; \partial\Omega,
	\end{cases}	
\end{equation}
where $\mathcal{L}_{\alpha}$ is a linear differential operator depends on $\alpha$. $\alpha=\{\alpha_1,\alpha_2,...\}$ is a finite set consisting of several parameters which can be scalars, vectors, functions, etc.. Other boundary conditions are also available.

Without loss of generality we take the well-known Poisson equation as an example. Consider
\begin{equation}\label{eq:poisson}
	\begin{cases}
		-\nabla\cdot(k\nabla u)=f &\quad \mbox{in} \; \Omega,
		\\
		u = 0 &\quad \mbox{on} \; \partial\Omega,
	\end{cases}	
\end{equation}
where $\Omega\subset\R^d$ is a bounded open set. Here $f\in L^2(\Omega), k\in L^\infty(\Omega)$, and $k\geq C$ a.e. for a constant $C>0$, so that it is a classical second-order elliptic equation. The variational formula is
\begin{equation}\label{eq:poisson_var}
\begin{split}
&{\rm Finding\ }u\in H_0^1(\Omega){\rm\ such\ that} \\
&\int_{\Omega}k\nabla u\cdot\nabla v dx=\int_{\Omega}fvdx,\quad \forall v\in H_0^1(\Omega). \\
\end{split}
\end{equation}
Now assume that $\Omega$ is a bounded polyhedral domain, $\mathcal{T}_h$ is a quasi-uniform and shape regular triangulation of $\Omega$. We apply the simple linear element space
\begin{equation}
V_h=\{v\in C(\overline{\Omega}):v|_K\in P_1(K), \forall K\in\mathcal{T}_h\}\subset V:=H_0^1(\Omega),
\end{equation}
where $P_1(K)$ denotes the function space consisting of all the linear polynomials over $K$. The finite element method is
\begin{equation}\label{eq:fem}
\begin{split}
&{\rm Finding\ }u_h\in V_h{\rm\ such\ that} \\
&\int_{\Omega}k\nabla u_h\cdot\nabla v_h dx=\int_{\Omega}fv_hdx,\quad \forall v_h\in V_h. \\
\end{split}
\end{equation}
Let $\{\phi_i\}_{i=1}^{n_h}$ be the nodal basis functions corresponding to all the interior nodes of $\mathcal{T}_h$. Then we obtain a discrete linear system for Eq. (\ref{eq:fem}) as
\begin{equation}\label{eq:discrete_sys}
A_h\mu_h=b_h,
\end{equation}
where $A_h=(\int_{\Omega}k\nabla \phi_i\cdot\nabla \phi_j dx)_{n_h\times n_h}$, $b_h=(\int_{\Omega}f\phi_idx)_{n_h\times 1}$. Denote $\mu_h=(\alpha_1,...,\alpha_{n_h})^T$, then $u_h:=\sum_{i=1}^{n_h}\alpha_i\phi_i$ is the numerical solution obtained by the above system. What we next to do is solving Eq. (\ref{eq:discrete_sys}) via a suitable iterative method.

For convenience, we temporarily omit the ``$h$'' in above notations, but we keep in mind that they indeed depend on $h$. For the discrete linear system

\begin{equation}\label{eq:discrete_sys_no_h}
A\mu=b,
\end{equation}
we consider the iterative form
\begin{equation}
\mu^{(m+1)}=\mu^{(m)}+B(b-A\mu^{(m)}),\quad \mu^{(0)}=0,
\end{equation}
where $B$ is an approximation of $A^{-1}$. The frequently-used choices of $B$ are

\begin{numcases}
{B=}
(\omega^{-1} I)^{-1}, & Richardson\label{eq:Richardson} \\
(\omega^{-1} D)^{-1}, & Jacobi (damped)\label{eq:Jacobi} \\
(\omega^{-1} D + L)^{-1}, & Gauss-Seidel (SOR)\label{eq:GS}
\end{numcases}
where $A=D+L+U$, $D$ is the diagonal of $A$, $L$ and $U$ are the strictly lower and upper triangular
parts of $A$, respectively. The $\omega$ is usually considered as the relaxation parameter to be chosen. By performing a suitable iterative method, we will get a numerical solution $\mu$ for system (\ref{eq:discrete_sys_no_h}).

Briefly, our hybrid iterative method is just inserting some inference steps via MIONet into the iterative process. Now we return to the operator regression model MIONet. For the problem (\ref{eq:poisson}), what we are concerned with is in fact the solution operator
\begin{equation}
\mathcal{G}:(k,f)\mapsto u.
\end{equation}
Now assume that we have a dataset as $\mathcal{T}=\{(k_i,f_i),u_i\}_{i=1}^N$ satisfying $\mathcal{G}(k_i,f_i)=u_i$. We pre-train a MIONet, denoted by $\mathcal{M}$, on this dataset with loss function
\begin{equation}
L(\theta)=\frac{1}{N}\sum_{i=1}^N\|\mathcal{M}(k_i,f_i;\theta)-u_i\|_{Y}^2,
\end{equation}
where $Y$ could be $C(\overline{\Omega})$, $L^2(\Omega)$, $H^1(\Omega)$, etc.. Note that here we set $\mathcal{M}$ linear with respect to $f$ due to Theorem \ref{thm:approximation}. An illustration is presented in Figure \ref{fig:architecture_MIONet_Poisson}. 
\begin{figure}[htbp]
    \centering
    \includegraphics[width=0.9\textwidth]{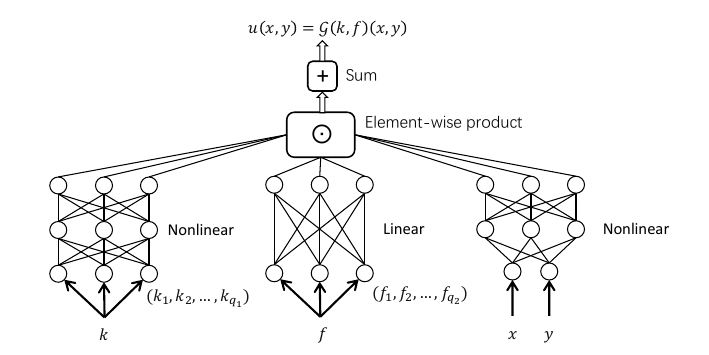}
    \caption{\textbf{An illustration of the architecture of MIONet for the 2-d Poisson equation.}}
    \label{fig:architecture_MIONet_Poisson}
\end{figure}
Denote $r^{(m)}:=b-A\mu^{(m)}$, which is the residual vector of $m$-th step. We expect to find a residual function related to the residual vector $r^{(m)}$. Define the residual function $\bar{r}^{(m)}:=\sum_{i=1}^{n}\beta_i\phi_i$, where $\beta=(\beta_1,...,\beta_n)^{T}$ satisfies
\begin{equation}\label{eq:sub_sys}
H\beta=r^{(m)},\quad H:=\left(\int_{\Omega}\phi_i\phi_jdx\right)_{n\times n}.
\end{equation}
Specially, we let $\bar{r}^{(0)}=f$. Note that $\bar{r}^{(m)}$ is linear with respect to $r^{(m)}$, we denote their relationship by a linear operator $\mathcal{H}$ as
\begin{equation}
\bar{r}^{(m)}=\mathcal{H}r^{(m)},\quad \mathcal{H}:=(\phi_1,...,\phi_n)H^{-1},\quad m\geq1.
\end{equation}
We further define the interpolation operator as
\begin{equation}
\mathbb{I}(u):=(u(x_1),...,u(x_{n}))^{T},
\end{equation}
where $x_i\in\Omega$ is the corresponding nodal coordinate of $\phi_i$, i.e. $\phi_i(x_j)=\delta_i^j$ (Dirac delta function).

Then the hybrid iterative method can be written as
\begin{numcases}
{\mu^{(m+1)}=}
\mu^{(m)}+\mathbb{I}(\mathcal{M}(k,\mathcal{H}(b-A\mu^{(m)}))), & $m\equiv 0{\rm\ mod\ }M$\label{eq:him_m}  \\
\mu^{(m)}+B(b-A\mu^{(m)}), & otherwise \label{eq:him_b}
\end{numcases}
where $M$ is a positive integer to be chosen, we regard it as the correction period. The algorithm is summarized in Algorithm \ref{alg:him}.

\begin{algorithm}
\caption{Hybrid iterative method for Poisson equation}
\label{alg:him}
\begin{algorithmic}
\REQUIRE{Parameters $k,f$, mesh $\mathcal{T}_h$, iterative method $B$, \\ pre-trained MIONet $\mathcal{M}$, and correction period $M$}
\ENSURE{Numerical solution $\mu$}
\STATE{1. Obtain a discrete system 
\begin{equation*}
A\mu=b
\end{equation*}
based on $\mathcal{T}_h$ for Eq. (\ref{eq:poisson}).}
\STATE{2. Initialize $\mu^{(1)}=\mathbb{I}(\mathcal{M}(k,f))$, $r^{(1)}=b-A\mu^{(1)}$, $m=1$}
\STATE{3. \textbf{while} ($\mu^{(m)}$ does not satisfy the convergence condition) \\
\ \quad\quad\quad\textbf{if} ($m\equiv 0{\rm\ mod\ }M$) \\
\ \quad\quad\quad\quad\quad$\mu^{(m+1)}=\mu^{(m)}+\mathbb{I}(\mathcal{M}(k,\mathcal{H}r^{(m)}))$ \\
\ \quad\quad\quad \textbf{else}\\
\ \quad\quad\quad\quad\quad$\mu^{(m+1)}=\mu^{(m)}+Br^{(m)}$\quad(in practice we update $\mu^{(m)}$ in-place without computing $B$)\\
\ \quad\quad\quad $r^{(m+1)}=b-A\mu^{(m+1)}$ \\
\ \quad\quad\quad $m=m+1$\quad\quad\quad\quad\quad\quad\quad\quad (set $m$ to $m+1$) \\
\ \quad\textbf{end}}
\STATE{4. \textbf{return} $\mu=\mu^{(m)}$}
\end{algorithmic}
\end{algorithm}

\textbf{Approximation and padding for residual function.} There is a new linear system (\ref{eq:sub_sys}) involved. We may choose an approximation for $\beta$, for example: 
\begin{equation}
\beta\approx {\rm diag}\left(\sum_{j}\int_{\Omega}\phi_i\cdot \phi_jdx\right)^{-1}r^{(m)}\approx{\rm diag}\left(\left(\int_{\Omega}\phi_i dx\right)^{-1}\right)r^{(m)}.
\end{equation}
With the solved $\beta$, we obtain the $\bar{r}^{(m)}=\mathcal{H}r^{(m)}$ expressed as a continuous piece-wise linear function w.r.t. $\{\phi_i\}$, which does not contain the boundary nodes. Hence the current $\bar{r}^{(m)}$ equals to 0 on $\partial\Omega$, which is regarded as a ``zero padding''. Here we may use other better padding for $\bar{r}^{(m)}$, for example, add boundary nodes whose values equal to the values of their corresponding nearest interior nodes (we can also apply other padding strategy like reflection padding, etc.). For convenience, we still use the notation $\bar{r}^{(m)}$ and $\mathcal{H}$, i.e.
\begin{equation}\label{eq:H_appro}
\bar{r}^{(m)}=\mathcal{H}r^{(m)}:={\rm Padding}\left((\phi_1,...,\phi_n){\rm diag}\left(\left(\int_{\Omega}\phi_1 dx\right)^{-1},...,\left(\int_{\Omega}\phi_n dx\right)^{-1}\right)r^{(m)}\right).
\end{equation}
Note that the above approximation brings an acceptable error for this hybrid method. Since $\bar{r}^{(0)}=f$, $\mu^{(1)}$ is exactly the prediction by MIONet. We may also think of $\mu^{(1)}$ as an initial guess provided by MIONet for the iterative method. When feeding $\bar{r}^{(m)}$ to $\mathcal{M}$, we need to evaluate the linear element function $\bar{r}^{(m)}$ at some sampling points (usually with smaller scale compared to the number of nodes $n$) for $\mathcal{M}$.

\textbf{Efficient interpolation.} The interpolation $\mathbb{I}$ is easy to compute, since the prediction
\begin{equation}
\mathcal{M}(k,\bar{r}^{(m)})(\cdot)
\end{equation}
is expressed by a neural network (encoded as a trunk net) rather than a discrete vector, and we only need to directly feed the interpolated points into the trunk net of the ONet. This is indeed an important advantage of the ONets series architectures.

\textbf{Hybrid iteration with multigrid method.} Since there is no requirement for the basic iterative method in this algorithm, we can also utilize other advanced iterative methods such as the multigrid method here. We will show more details in the experiment part.

\textbf{Inhomogeneous boundary condition.} Here we further investigate the inhomogeneous boundary condition. Consider
\begin{equation}\label{eq:inhomo}
	\begin{cases}
		-\nabla\cdot(k\nabla u)=f &\quad \mbox{in} \; \Omega,
		\\
		Tu = g &\quad \mbox{on} \; \partial\Omega,
	\end{cases}	
\end{equation}
where $T$ is the trace operator. In such a case, we aim to learn an operator mapping from $(k,f,g)$ to the solution $u$, i.e.,
\begin{equation}
\mathcal{G}:(k,f,g)\mapsto u.
\end{equation}
Of course we can train an MIONet $\mathcal{M}(k,f,g)$ to learn this operator. However, this strategy losses the linearity w.r.t. the residual vector/function during the iteration, and we strive to preserve this linearity to guarantee the convergency of the hybrid method. Denote by $E$ the trace extension operator for (\ref{eq:inhomo}), then $T(Eg)=g$ for $g\in H^{1/2}(\partial\Omega)$. We derive that
\begin{equation}
	\begin{cases}
		-\nabla\cdot(k\nabla u_0)=f+\nabla\cdot(k\nabla Eg) &\quad \mbox{in} \; \Omega,
		\\
		Tu_0 = 0 &\quad \mbox{on} \; \partial\Omega,
	\end{cases}
\end{equation}
where $u=u_0+Eg$. Therefore the model can be modified to
\begin{equation}
\tilde{\mathcal{M}}(k,f,g):=\mathcal{M}(k, f+\nabla\cdot(k\nabla Eg)) + Eg,
\end{equation}
where $\mathcal{M}$ is a standard MIONet which is linear w.r.t. the second input as before. Although the trace extension operator $E$ is unique defined, we may choose a suitable constructed extension as needed, for example the continuous piece-wise linear function related to a grid, which takes 0 at the interior nodes and matches $g$ at the boundary nodes. Note that here we regard $f+\nabla\cdot(k\nabla Eg)$ as an element in $H^{-1}(\Omega)$.

There is in fact another way which is much easier to achieve. It is not difficult to find that $\mathcal{G}(k,f,g)$ is linear w.r.t. $(f,g)\in L^2(\Omega)\times H^{1/2}(\partial\Omega)$, i.e.,
\begin{equation}
\mathcal{G}(k,(\lambda_1(f_1,g_1)+\lambda_2(f_2,g_2)))=\lambda_1\mathcal{G}(k,(f_1,g_1))+\lambda_2\mathcal{G}(k,(f_2,g_2)),
\end{equation}
then we naturally apply the model
\begin{equation}
\tilde{\mathcal{M}}(k,f,g):=\mathcal{M}(k,(f,g)),
\end{equation}
where $(f,g)$ is a simple cartesian product of $f$ and $g$, and it will be fed into the linear branch net of $\mathcal{M}$. In such a case, the discrete algebraic system of (\ref{eq:inhomo}) is augmented with several rows and columns taking into account the boundary nodes, where the boundary block is the identity. The unknown and residual vectors contain not only the coefficients of the interior nodes but also the boundary nodes. We apply this strategy in our experiment.

\textbf{Necessary convergence condition for operator regression models.} There is a necessary condition for the convergence of the hybrid iterative method. Assume that the hybrid method converges with a model $\mathcal{M}$, then
\begin{equation}
\lim_{s\to\infty}\mu^{(sM+1)}=\lim_{s\to\infty}\left(\mu^{(sM)}+\mathbb{I}(\mathcal{M}(k,\mathcal{H}r^{(sM)}))\right),
\end{equation}
and we have
\begin{equation}
\mathbb{I}(\mathcal{M}(k,0))=0.
\end{equation}
As the interpolation operator $\mathbb{I}$ is related to the mesh $\mathcal{T}_h$, and the above equation holds for arbitrary $\mathbb{I}_{\mathcal{T}_h}$, we know
\begin{equation}\label{eq:necessary_condition}
\mathcal{M}(k,0)=0.
\end{equation}
Obviously, Eq. (\ref{eq:necessary_condition}) holds as long as $\mathcal{M}$ is linear w.r.t. the second input. 

\section{Theoretical analysis}\label{sec:theory}
Consider the one-dimensional two-point boundary value problem of system (\ref{eq:poisson}) where $k=1$, $\Omega=(0, 1)$, i.e.,
\begin{equation}\label{eq:poisson_1d}
	\begin{cases}
		-u''=f,\quad x\in(0,1),
		\\
		u(0)=u(1)=0.
	\end{cases}	
\end{equation}
We choose a uniform grid with size $h=\frac{1}{n+1}$ and apply the linear element bases functions $\{\phi_i\}_{i=1}^{n}$ whose nodal points are interior, then the above system leads to a discrete system as
\begin{equation}\label{eq:1d_sys}
A\mu=b,
\end{equation}
where
\begin{equation}
A=\frac{1}{h}\begin{pmatrix}
2 & -1 & & & \\
-1 & 2  & -1 & & \\
 & \ddots & \ddots & \ddots & \\
 &   & -1 & 2 & -1 \\
 &  & & -1 & 2 \\
\end{pmatrix},\quad \mu=\begin{pmatrix}
\mu_1 \\
\mu_2 \\
\vdots \\
\mu_{n-1} \\
\mu_{n}\\
\end{pmatrix},\quad b=\begin{pmatrix}
b_1 \\
b_2 \\
\vdots \\
b_{n-1} \\
b_{n}\\
\end{pmatrix},
\end{equation}
and $b_i=\int_{0}^{1}f\phi_idx$. We can easily write out the eigenvalues and the eigenvectors of $A$ as
\begin{equation}
\lambda_i=\frac{4}{h}\sin^2(\frac{\pi}{2}hi),\quad \xi^i=(\xi^i_1,...,\xi^i_n)^T,\quad \xi^i_j=\sqrt{2h}\sin(i\pi hj),\quad 1\leq i,j\leq n.
\end{equation}
Denote $\Xi=(\xi^1,...,\xi^n)$, then $\Xi$ is symmetric and orthogonal, i.e. $\Xi^T=\Xi$ and $\Xi^{-1}=\Xi$.

\subsection{Convergence condition}\label{sec:convergence}
Recall that we apply the hybrid iterative method (\ref{eq:him_m}-\ref{eq:him_b}) to solve (\ref{eq:1d_sys}). For convenience, here we simply set $\bar{r}^{(0)}=\mathcal{H}b=\mathcal{H}r^{(0)}$, so that we have $\bar{r}^{(m)}=\mathcal{H}r^{(m)}$ for all $m\geq 0$. We further denote that
\begin{equation}
\bar{\mathcal{M}}(r):=\mathbb{I}(\mathcal{M}(1,\mathcal{H}r)),
\end{equation}
thus $\bar{\mathcal{M}}:\R^n\to\R^n$ is a linear map. Denote that
\begin{equation}
e^{(m)}:=\mu - \mu^{(m)},
\end{equation}
then we find that
\begin{equation}
e^{(1)}=\mu-\bar{\mathcal{M}}(r^{(0)})=(I-\bar{\mathcal{M}}A)e^{(0)},
\end{equation}
and for $1\leq m\leq M-1$,
\begin{equation}
e^{(m+1)}=(I-BA)e^{(m)},
\end{equation}
thus
\begin{equation}
e^{(M)}=(I-BA)^{M-1}e^{(1)}=(I-BA)^{M-1}(I-\bar{\mathcal{M}}A)e^{(0)}.
\end{equation}
Consequently we have
\begin{equation}
e^{(sM)}=U^se^{(0)},
\end{equation}
where
\begin{equation}
U:=(I-BA)^{M-1}(I-\bar{\mathcal{M}}A).
\end{equation}
Up to now, we have derived the convergence results of the hybrid iterative method.
\begin{theorem}
The hybrid iterative method (\ref{eq:him_m}-\ref{eq:him_b}) converges if and only if
\begin{equation}
\rho\left((I-BA)^{M-1}(I-\bar{\mathcal{M}}A)\right)<1,
\end{equation}
where $\rho(\cdot)$ denotes the spectral radius.
\end{theorem}

\begin{corollary}
There exists an $M$ large enough such that the hybrid iterative method (\ref{eq:him_m}-\ref{eq:him_b}) converges, if the applied original iterative method converges, i.e.
\begin{equation}
\rho(I-BA)<1.
\end{equation}
\end{corollary}
\noindent The corollary is also obvious since $(I-BA)^{M-1}(I-\bar{\mathcal{M}}A)\to 0$ as $M\to\infty$ if $\rho(I-BA)<1$.


\begin{corollary}
The hybrid iterative method (\ref{eq:him_m}-\ref{eq:him_b}) with Richardson iteration, i.e. $B=\omega I$, converges if
\begin{equation}\label{eq:Ri_condition}
0<\omega<\frac{2}{\rho(A)},\quad M\geq2+\left\lfloor-\frac{\ln\|I-\bar{\mathcal{M}}A\|_2}{\ln\rho(I-\omega A)}\right\rfloor.
\end{equation}
\end{corollary}
\begin{proof}
If $0<\omega<\frac{2}{\rho(A)}$, we have
\begin{equation*}
\rho(I-BA)=\rho(I-\omega A)<1,
\end{equation*}
thus the Richardson iteration converges. Furthermore, if Eq. (\ref{eq:Ri_condition}) holds, then
\begin{equation*}
\begin{split}
\rho\left((I-BA)^{M-1}(I-\bar{\mathcal{M}}A)\right)&\leq\norm{(I-\omega A)^{M-1}(I-\bar{\mathcal{M}}A)}_2 \\
&\leq\norm{I-\omega A}_2^{M-1}\cdot\norm{I-\bar{\mathcal{M}}A}_2 \\
&=\rho(I-\omega A)^{M-1}\cdot\norm{I-\bar{\mathcal{M}}A}_2 \\
&<\rho(I-\omega A)^{1-\frac{\ln\|I-\bar{\mathcal{M}}A\|_2}{\ln\rho(I-\omega A)}-1}\cdot\norm{I-\bar{\mathcal{M}}A}_2 \\
&=1.
\end{split}
\end{equation*}
\end{proof}
It is noteworthy that there is no requirement on the loss of the MIONet $\mathcal{M}$ to guarantee the convergence. However, an $\mathcal{M}$ with a large loss cannot accelerate the iteration, which will be discussed later.

\subsection{Error of model inference}\label{sec:mionet_error}
Since the MIONet usually learns the low-frequency data, we assume that the dataset includes the previous several eigenfunctions as 
\begin{equation}\label{eq:dataset}
\mathcal{T}=\{(1,\bar{\lambda}_i\bar{\xi}^i),\bar{\xi}^i\}_{i=1}^{n_0},\quad 1<n_0\leq n,
\end{equation}
where $\bar{\lambda}_i,\bar{\xi}^i$ are the eigenvalue and the eigenfunction of the original system, i.e.,
\begin{equation}
\bar{\lambda}_i=i^2\pi^2,\quad \bar{\xi}^i(x)=\sin(i\pi x),\quad \quad 1\leq i\leq n.
\end{equation}
For convenience, we consider the loss
\begin{equation}\label{eq:loss}
L(\theta)=\frac{1}{n_0}\sum_{i=1}^{n_0}\ell_i(\theta):=\frac{1}{n_0}\sum_{i=1}^{n_0}\|\mathcal{M}(1,\bar{\lambda}_i\bar{\xi}^i;\theta)-\bar{\xi}^i\|_{C[0,1]}^2.
\end{equation}
We further impose a generalization assumption on the MIONet for this case.
\begin{assumption}[generalization assumption]\label{ass:generalization}
Assume that the MIONet keeps a consistent Lipschitz constant $L>0$ w.r.t. the second input during training, i.e.
\begin{equation}
\norm{\mathcal{M}(k,f_1;\theta)-\mathcal{M}(k,f_2;\theta)}_{C[0,1]}\leq L\norm{f_1-f_2}_{C[0,1]}
\end{equation}
holds for $f_1,f_2\in C[0,1]$ and $\theta\in\{\theta_1,\theta_2,...\}$. The set $\{\theta_t\}$ is regarded as the sequence of parameters during the training process. $L$ depends only on $k$.
\end{assumption}

Let $\mathcal{H}$ be defined as in Eq. (\ref{eq:H_appro}) and extended with replication padding. Then
\begin{equation}
\mathcal{H}(\sqrt{\frac{h}{2}}\xi^i)=(\phi_0, \phi_1,...,\phi_n,\phi_{n+1})\cdot \sqrt{\frac{1}{2h}}(\xi_1^i,\xi_1^i,\xi_2^i,...,\xi_{n-1}^i,\xi_n^i,\xi_n^i)^T,
\end{equation}
where $\phi_0$ and $\phi_{n+1}$ are the two boundary nodal functions, hence
\begin{equation}
\begin{split}
&\norm{\mathcal{H}(\sqrt{\frac{h}{2}}\xi^i)-\bar{\xi}^i}_{C[0,1]} \\
=&\max\big(\norm{\sin(i\pi h)-\bar{\xi}^i}_{C[0,h]},\norm{\sin(i\pi h)\phi_1+\sin(2i\pi h)\phi_2-\bar{\xi}^i}_{C[h,2h]},...,\\
&\norm{\sin((n-1)i\pi h)\phi_{n-1}+\sin(ni\pi h)\phi_n-\bar{\xi}^i}_{C[(n-1)h,nh]},\norm{\sin(ni\pi h)-\bar{\xi}^i}_{C[nh,1]}\big) \\
\leq&i\pi h.
\end{split}
\end{equation}
Consequently,
\begin{equation}
\begin{split}
\norm{\bar{\mathcal{M}}(\sqrt{\frac{h}{2}}\xi^i)-\mathbb{I}(\mathcal{M}(1,\bar{\xi}^i))}_{\infty}=&\norm{\mathbb{I}(\mathcal{M}(1,\mathcal{H}(\sqrt{\frac{h}{2}}\xi^i)-\bar{\xi}^i))}_{\infty} \\
\leq&\norm{\mathcal{M}(1,\mathcal{H}(\sqrt{\frac{h}{2}}\xi^i)-\bar{\xi}^i)}_{C[0,1]} \\
\leq&L_1\norm{\mathcal{H}(\sqrt{\frac{h}{2}}\xi^i)-\bar{\xi}^i}_{C[0,1]} \\
\leq&L_1i\pi h,
\end{split}
\end{equation}
where $L_1$ is the Lipschitz constant defined as in Assumption \ref{ass:generalization}.
By Eq. (\ref{eq:loss}),
\begin{equation}\label{eq:pre_estimate}
\begin{split}
\norm{\bar{\mathcal{M}}(\bar{\lambda}_i\sqrt{\frac{h}{2}}\xi^i)-\sqrt{\frac{1}{2h}}\xi^i}_{\infty}=&\norm{\bar{\mathcal{M}}(\bar{\lambda}_i\sqrt{\frac{h}{2}}\xi^i)-\mathbb{I}(\mathcal{M}(1,\bar{\lambda}_i\bar{\xi}^i))+\mathbb{I}(\mathcal{M}(1,\bar{\lambda}_i\bar{\xi}^i))-\mathbb{I}(\bar{\xi}^i)}_{\infty} \\
\leq&\bar{\lambda}_i\norm{\bar{\mathcal{M}}(\sqrt{\frac{h}{2}}\xi^i)-\mathbb{I}(\mathcal{M}(1,\bar{\xi}^i))}_{\infty} +\norm{\mathbb{I}(\mathcal{M}(1,\bar{\lambda}_i\bar{\xi}^i))-\mathbb{I}(\bar{\xi}^i)}_{\infty} \\
\leq&\bar{\lambda}_iL_1i\pi h+\ell_i \\
=&L_1\pi^3i^3h+\ell_i.
\end{split}
\end{equation}
Now denote
\begin{equation}
e_i:=\xi^i-\bar{\mathcal{M}}(\lambda_i\xi^i),\quad1\leq i\leq n,
\end{equation}
and we get the estimate
\begin{equation}\label{eq:estimate_detail_1}
\begin{split}
\norm{e_i}_2=&\norm{\xi^i-\bar{\mathcal{M}}(\lambda_i\xi^i)}_2 \\
=&\frac{\lambda_i}{\bar{\lambda}_i}\sqrt{\frac{2}{h}}\norm{\bar{\mathcal{M}}(\bar{\lambda}_i\sqrt{\frac{h}{2}}\xi^i)-\sqrt{\frac{1}{2h}}\xi^i+\sqrt{\frac{1}{2h}}\xi^i-\frac{\bar{\lambda}_i}{\lambda_i}\sqrt{\frac{h}{2}}\xi^i}_2  \\
\leq&\frac{\lambda_i}{\bar{\lambda}_i}\sqrt{\frac{2}{h}}\left(\norm{\bar{\mathcal{M}}(\bar{\lambda}_i\sqrt{\frac{h}{2}}\xi^i)-\sqrt{\frac{1}{2h}}\xi^i}_2 + \norm{\sqrt{\frac{1}{2h}}\xi^i-\frac{\bar{\lambda}_i}{\lambda_i}\sqrt{\frac{h}{2}}\xi^i}_2\right) \\
\leq&\frac{\lambda_i}{\bar{\lambda}_i}\sqrt{\frac{2}{h}}\left(\sqrt{n}\norm{\bar{\mathcal{M}}(\bar{\lambda}_i\sqrt{\frac{h}{2}}\xi^i)-\sqrt{\frac{1}{2h}}\xi^i}_{\infty} + \norm{\sqrt{\frac{1}{2h}}\xi^i-\frac{\bar{\lambda}_i}{\lambda_i}\sqrt{\frac{h}{2}}\xi^i}_2\right). \\
\end{split}
\end{equation}
With the results of (\ref{eq:pre_estimate}-\ref{eq:estimate_detail_1}) and the orthogonality of $\Xi$, we further derive that
\begin{equation}
\begin{split}
\norm{e_i}_2\leq&\frac{\lambda_i}{\bar{\lambda}_i}\sqrt{\frac{2}{h}}\left(\sqrt{n}(L_1\pi^3i^3 h+\ell_i) + \norm{\sqrt{\frac{1}{2h}}\xi^i-\frac{\bar{\lambda}_i}{\lambda_i}\sqrt{\frac{h}{2}}\xi^i}_2\right) \\
=&\frac{4\sin^2(\frac{\pi}{2}hi)}{\pi^2i^2h^2}\sqrt{2(1-h)}(L_1\pi^3i^3h+\ell_i)+\left|\frac{4\sin^2(\frac{\pi}{2}hi)}{\pi^2i^2h^2}-1\right|.
\end{split}
\end{equation}
It is elementary to see
\begin{equation}
1-\frac{1}{3}x^2\leq\frac{\sin^2(x)}{x^2}\leq1,\quad\forall 0\neq x\in \mathbb{R}, 
\end{equation}
so that
\begin{equation}\label{eq:estimate_detail}
\begin{split}
\norm{e_i}_2\leq&\frac{4\sin^2(\frac{\pi}{2}hi)}{\pi^2i^2h^2}\sqrt{2(1-h)}(L_1\pi^3i^3h+\ell_i)+\left|\frac{4\sin^2(\frac{\pi}{2}hi)}{\pi^2i^2h^2}-1\right| \\
\leq&\sqrt{2}L_1\pi^3i^3 h+\sqrt{2}\ell_i+1-(1-\frac{1}{3}(\frac{\pi}{2}hi)^2)\\
=&\sqrt{2}L_1\pi^3i^3 h + \frac{\pi^2}{12}i^2h^2 +\sqrt{2}\ell_i, \quad 1\leq i\leq n_0.
\end{split}
\end{equation}
It is summarized as
\begin{equation}\label{eq:estimate_h}
\norm{e_i}_2=\mathcal{O}(i^3h) + \mathcal{O}(\ell_i),\quad 1\leq i\leq n_0.
\end{equation}
The error includes two parts, and the terms $\mathcal{O}(i^3h),\mathcal{O}(\ell_i)$ are caused by the errors of the discretization and the loss, respectively. The first part of this error is inevitable even though the loss $\ell_i$ has been trained to zero, since there exists a gap between the eigenvector (eigenvalue) of the discrete system (\ref{eq:1d_sys}) and the eigenfunction (eigenvalue) of the original system (\ref{eq:poisson_1d}), which becomes larger as $i$ increases.

The estimate (\ref{eq:estimate_h}) can be equivalently written as
\begin{equation}\label{eq:estimate_n}
\norm{e_i}_2=\mathcal{O}\left(\frac{i^3}{n}\right) + \mathcal{O}(\ell_i),\quad 1\leq i\leq n_0.
\end{equation}
Eq. (\ref{eq:estimate_n}) tells that $i\ll n$ is necessary to ensure that $\norm{e_i}_2$ is small enough. Therefore $n_0$, the size of the dataset $\mathcal{T}$ defined in (\ref{eq:dataset}), is no need to be too large compared to $n$. We reasonably assume that $n_0\ll n$. For a fixed $1\leq i\leq n_0$, we have derived that
\begin{equation}
\lim_{h\to 0 ,\ell_i\to 0}\norm{e_i}_2=0,\quad 1\leq i\leq n_0.
\end{equation}

Note that here we consider the ground truth, i.e. the analytical eigenvectors and eigenvalues, as data. If we utilize numerical solutions as data, the error will decrease, since the gap between data and the discrete system becomes smaller. As a matter of fact, if we replace the $\bar{\lambda}_i$ in the dataset (\ref{eq:dataset}) by the corresponding numerical eigenvalue $\frac{\lambda_i}{h}$, then the error term $\frac{\pi^2}{12}i^2h^2$ in the last equality of (\ref{eq:estimate_detail}) will disappear.



\subsection{Spectral analysis}
With the above discussion, we are able to analyze the spectral behavior and the convergence speed of the proposed hybrid iterative method. Now we express $e^{(0)},e^{(1)}$ (see Section \ref{sec:convergence}) by the eigenvectors of $A$ as
\begin{equation}
	\begin{cases}
		e^{(0)}=\mu-\mu^{(0)}=\mu=\sum_{i=1}^{n}\alpha_i\xi^i=\Xi\cdot\alpha,
		\\
		e^{(1)}=(I-\bar{\mathcal{M}}A)e^{(0)}=\sum_{i=1}^{n}\alpha_i(I-\lambda_i \bar{\mathcal{M}})\xi^i=:\sum_{i=1}^{n}\beta_i\xi^i=\Xi\cdot\beta,
	\end{cases}	
\end{equation}
where $\alpha:=(\alpha_1,...,\alpha_n)^T$, $\beta:=(\beta_1,...,\beta_n)^T$. We further set
\begin{equation}\label{eq:YZ}
	\begin{cases}
		Y:=\Xi^{-1}\bar{\mathcal{M}}\Xi=(\xi^1,...,\xi^n)^{-1}\bar{\mathcal{M}}(\xi^1,...,\xi^n),
		\\
		Z:=\Xi^{-1}B\Xi=(\xi^1,...,\xi^n)^{-1}B(\xi^1,...,\xi^n),
	\end{cases}	
\end{equation}
then by direct computing we have
\begin{equation}
\beta=(I-Y\Lambda)\alpha,\quad \Lambda={\rm diag}(\lambda_1,...,\lambda_n),
\end{equation}
so that
\begin{equation}
e^{(1)}=\Xi\cdot(I-Y\Lambda)\alpha.
\end{equation}
Similarly, we can derive that
\begin{equation}
e^{(M)}=\Xi\cdot(I-Z\Lambda)^{M-1}(I-Y\Lambda)\alpha.
\end{equation}
Consequently, we obtain the final error as
\begin{equation}
e^{(sM)}=\Xi\cdot T^s\alpha,
\end{equation}
where
\begin{equation}\label{eq:T}
T:=(I-Z\Lambda)^{M-1}(I-Y\Lambda).
\end{equation}

Next we proceed discussions on $T$. Note that the $Y$ defined in (\ref{eq:YZ}) depends on the trained operator regression model MIONet, i.e., $\mathcal{M}$. Assume that $\bar{\mathcal{M}}$ perfectly learns the eigenvectors of Eq. (\ref{eq:1d_sys}), so that
\begin{equation}
A\bar{\mathcal{M}}(\xi^i)=\xi^i,
\end{equation}
which leads to
\begin{equation}
\bar{\mathcal{M}}(\xi^i)=\lambda_i^{-1}\xi^i.
\end{equation}
Then
\begin{equation}
Y=\Xi^{-1}\bar{\mathcal{M}}\Xi=\Lambda^{-1},
\end{equation}
thus
\begin{equation}
Y\Lambda=I,\quad T=0.
\end{equation}
In such a situation, the first iterative step by MIONet in fact gives the right solution for the discrete system (\ref{eq:1d_sys}). However, we have to take into account the error of MIONet. Keep in mind that $Y$ is in fact an approximation of $\Lambda^{-1}$, therefore we expect
\begin{equation}
I-Y\Lambda\approx0,
\end{equation}
at least for the first few columns. As discussed in Section \ref{sec:mionet_error}, we recall that
\begin{equation}
e_i=\xi^i-\bar{\mathcal{M}}(\lambda_i\xi^i),\quad 1\leq i\leq n,
\end{equation}
and
\begin{equation}
\norm{e_i}_2=\mathcal{O}(i^3h) + \mathcal{O}(\ell_i),\quad 1\leq i\leq n_0.
\end{equation}
The errors of MIONet on the previous $n_0$ eigenvectors can be controlled by $h$ and the loss $\ell_i$. Hence we assume that $e_1,...,e_{n_0}$ are small enough with sufficiently training for MIONet, while $e_{n_0+1},...,e_n$ are uncontrolled. Therefore,
\begin{equation}
\begin{split}
I-Y\Lambda=&\Xi^{-1}\cdot(\Xi-\bar{\mathcal{M}}\Xi\Lambda) \\
=&\Xi\cdot(e_1,...,e_{n_0},e_{n_0+1},...,e_n) \\
=&\begin{bmatrix}
\Xi\cdot(e_1,...,e_{n_0}) & \Xi\cdot(e_{n_0+1},...,e_n)
\end{bmatrix},
\end{split}
\end{equation}
whose first $n_0$ columns are small enough. It indicates that a well-trained MIONet can eliminate the low-frequency components of the error.

Return to the final error $e^{(sM)}$, we have
\begin{equation}
\begin{split}
e^{(sM)}&=\Xi\cdot((I-Z\Lambda)^{M-1}(I-Y\Lambda))^s\alpha \\
&=\Xi\cdot(I-Z\Lambda)^{M-1}\left((I-Y\Lambda)(I-Z\Lambda)^{M-1}\right)^{s-1}(I-Y\Lambda)\alpha \\
&=\Xi\cdot(I-Z\Lambda)^{M-1}\tilde{T}^{s-1}(I-Y\Lambda)\alpha,
\end{split}
\end{equation}
where
\begin{equation}
\tilde{T}:=(I-Y\Lambda)(I-Z\Lambda)^{M-1}.
\end{equation}
We subsequently study how the $\tilde{T}$ acts on a vector $v=(v_1,...,v_n)^T\in\R^n$. Now assume that we choose the Richardson iteration and set $\omega=\frac{h}{4}$ which satisfies the convergence condition, then $Z=\frac{h}{4}I$ and
\begin{equation}
\begin{split}
\tilde{T}v&=(I-Y\Lambda)(I-\frac{h}{4}\Lambda)^{M-1}v \\
&=\Xi^{-1}\cdot(\Xi-\bar{\mathcal{M}}\Xi\Lambda)\left(\cos^{2(M-1)}(\frac{\pi}{2}h)v_1,..., \cos^{2(M-1)}(\frac{\pi}{2}hn)v_n\right)^T \\
&=\Xi\cdot(e_1,...,e_n)\left(\cos^{2(M-1)}(\frac{\pi}{2}h)v_1,..., \cos^{2(M-1)}(\frac{\pi}{2}hn)v_n\right)^T \\
&=\Xi\cdot\left(\sum_{i=1}^{n}\cos^{2(M-1)}(\frac{\pi}{2}hi)v_ie_i\right),
\end{split}
\end{equation}
therefore
\begin{equation}
\begin{split}
\norm{\tilde{T}v}_2&=\norm{\sum_{i=1}^{n}\cos^{2(M-1)}(\frac{\pi}{2}hi)v_ie_i}_2 \\
&\leq \sum_{i=1}^{n}\cos^{2(M-1)}(\frac{\pi}{2}hi)|v_i|\norm{e_i}_2 \\
&\leq \left(\sum_{i=1}^{n}\cos^{4(M-1)}(\frac{\pi}{2}hi)\norm{e_i}_2^2\right)^{\frac{1}{2}}\left(\sum_{i=1}^{n}|v_i|^2\right)^{\frac{1}{2}} \\
&\leq\left(\sum_{i=1}^{n}\cos^{2(M-1)}(\frac{\pi}{2}hi)\norm{e_i}_2\right)\norm{v}_2.
\end{split}
\end{equation}
As
\begin{equation}
	\begin{cases}
		\cos^2(\frac{\pi}{2}hi)\leq\cos^2(\frac{\pi}{2}h),\quad 1\leq i\leq n_0,
		\\
		\cos^2(\frac{\pi}{2}hi)\leq\cos^2(\frac{\pi}{2}h(n_0+1)),\quad n_0+1\leq i\leq n,
	\end{cases}	
\end{equation}
we obtain the estimate
\begin{equation}
\norm{\tilde{T}v}_2\leq \left(\cos^{2(M-1)}(\frac{\pi}{2}h)\sum_{i=1}^{n_0}\norm{e_i}_2+\cos^{2(M-1)}(\frac{\pi}{2}h(n_0+1))\sum_{i=n_0+1}^{n}\norm{e_i}_2\right)\norm{v}_2,
\end{equation}
which shows the change of the error within a period of $M$ steps. Subsequently, we write down the total error as
\begin{equation}
\begin{split}
&\norm{e^{(sM)}}_2 \\
=&\norm{(I-Z\Lambda)^{M-1}\tilde{T}^{s-1}(I-Y\Lambda)\alpha}_2 \\
=&\norm{(I-\frac{h}{4}\Lambda)^{M-1}\tilde{T}^{s-1}(I-\Xi^{-1}\bar{\mathcal{M}}\Xi\Lambda)\alpha}_2 \\
\leq&\cos^{2(M-1)}(\frac{\pi}{2}h)\norm{\tilde{T}^{s-1}(I-\Xi^{-1}\bar{\mathcal{M}}\Xi\Lambda)\alpha}_2 \\
\leq& \cos^{2(M-1)}(\frac{\pi}{2}h)\left(\cos^{2(M-1)}(\frac{\pi}{2}h)\sum_{i=1}^{n_0}\norm{e_i}_2+\cos^{2(M-1)}(\frac{\pi}{2}h(n_0+1))\sum_{i=n_0+1}^{n}\norm{e_i}_2\right)^{s-1} \\
&\cdot\norm{(I-\Xi^{-1}\bar{\mathcal{M}}\Xi\Lambda)\alpha}_2.
\end{split}
\end{equation}
Moreover, the last term in above inequality is
\begin{equation}
\begin{split}
&\norm{(I-\Xi^{-1}\bar{\mathcal{M}}\Xi\Lambda)\alpha}_2\\=&\norm{(\Xi-\bar{\mathcal{M}}\Xi\Lambda)\alpha}_2 \\
=&\norm{(e_1,...,e_n)\alpha}_2 \\
\leq&\sum_{i=1}^{n_0}|\alpha_i|\norm{e_i}_2+\sum_{i=n_0+1}^{n}|\alpha_i|\norm{e_i}_2 \\
\leq&C\left(\cos^{2(M-1)}(\frac{\pi}{2}h)\sum_{i=1}^{n_0}\norm{e_i}_2+\cos^{2(M-1)}(\frac{\pi}{2}h(n_0+1))\sum_{i=n_0+1}^{n}\norm{e_i}_2\right),
\end{split}
\end{equation}
where $C>0$ is a constant independent of $s$. We hence have the simple relationship that
\begin{equation}
\norm{e^{(sM)}}_2\leq C\cdot\left(\cos^{2(M-1)}(\frac{\pi}{2}h)\sum_{i=1}^{n_0}\norm{e_i}_2+\cos^{2(M-1)}(\frac{\pi}{2}h(n_0+1))\sum_{i=n_0+1}^{n}\norm{e_i}_2\right)^s.
\end{equation}
For comparison, we write down the error of the original Richardson method as
\begin{equation}
\norm{e^{(sM)}_{\rm Ri}}_2=\left(\sum_{i=1}^n\cos^{4sM}(\frac{\pi}{2}hi)\alpha_i^2\right)^{\frac{1}{2}}\geq \cos^{2sM}(\frac{\pi}{2}h)|\alpha_1|=C_{\rm Ri}\cdot\left(\cos^{2M}(\frac{\pi}{2}h)\right)^s,
\end{equation}
where $C_{\rm Ri}>0$ is a constant under the assumption that $\alpha_1\neq0$.

Now we show the errors of the Richardson iterative method and the corresponding hybrid method as
\begin{numcases}
{}
\norm{e^{(sM)}_{\rm Ri}}_2\geq C_{\rm Ri}\cdot\left(\eta_1^M\right)^s,\quad\quad\quad\quad & Richardson\label{eq:Richardson_error} \\
\norm{e^{(sM)}}_2\leq C\cdot\left(\epsilon\cdot\eta_1^{M-1}+R\cdot\eta_2^{M-1}\right)^s,\quad\quad\quad\quad & Hybrid\label{eq:Hybrid_error}
\end{numcases}
where
\begin{equation}
\eta_1=\cos^{2}(\frac{\pi}{2}h),\quad \eta_2=\cos^{2}(\frac{\pi}{2}h(n_0+1)),\quad \epsilon=\sum_{i=1}^{n_0}\norm{e_i}_2,\quad 
R=\sum_{i=n_0+1}^{n}\norm{e_i}_2.
\end{equation}
Note that $\eta_1\approx1$, $1>\eta_1>\eta_2>0$, and $\epsilon$ is usually small enough while $R$ is uncontrolled. Up to now, we are able to compare the convergence speeds of these two methods.
\begin{theorem}
Assume that $\epsilon<\eta_1$. Then there exists a $K>0$, such that for any integer $M>K$, it holds
\begin{equation}
\norm{e^{(sM)}}_2<\norm{e^{(sM)}_{\rm Ri}}_2
\end{equation}
for $s$ large enough. Briefly, the hybrid iterative method converges faster than the original Richardson iterative method for $M$ large enough as long as $\epsilon<\eta_1$.
\end{theorem}
\begin{proof}
If $\epsilon<\eta_1$, we have
\begin{equation}
\lim_{M\to\infty}\frac{\epsilon\cdot\eta_1^{M-1}+R\cdot\eta_2^{M-1}}{\eta_1^M}=\lim_{M\to\infty}\frac{\epsilon}{\eta_1}+\frac{R}{\eta_2}\cdot\left(\frac{\eta_2}{\eta_1}\right)^{M}=\frac{\epsilon}{\eta_1}<1.
\end{equation}
\end{proof}

Next we investigate how the $M$ influences the convergence rate of the hybrid iterative method. According to Eq. (\ref{eq:Hybrid_error}-\ref{eq:Richardson_error}), the average convergence rates of the above methods are derived as the following.
\begin{theorem}[convergence rate]
The convergence rate of the Richardson iterative method and an upper bound of the average convergence rate of the corresponding hybrid iterative method are
\begin{numcases}
{}
\lim_{s\to\infty}\norm{e^{(sM)}_{\rm Ri}}_2^{\frac{1}{sM}}= \eta_1,\quad\quad\quad\quad & {\rm Richardson} \\
\varlimsup_{s\to\infty}\norm{e^{(sM)}}_2^{\frac{1}{sM}}\leq\eta_1\cdot\left(\frac{\epsilon}{\eta_1}+\frac{R}{\eta_2}\cdot\left(\frac{\eta_2}{\eta_1}\right)^M\right)^{\frac{1}{M}}.\quad\quad\quad\quad & {\rm Hybrid}
\end{numcases}
\end{theorem}
\noindent This theorem clearly points out how the network correction step accelerates the iteration.

Denote
\begin{equation}\label{eq:convergence_rate}
{\rm Rate}(M)=\eta_1\cdot\left(\frac{\epsilon}{\eta_1}+\frac{R}{\eta_2}\cdot\left(\frac{\eta_2}{\eta_1}\right)^M\right)^{\frac{1}{M}}.
\end{equation}
To visually observe the tendency of Rate($M$), we for example set $\eta_1=0.999, \eta_2=0.5, \epsilon=0.1, R=10$, and plot them in Figure \ref{fig:rate_M}. It is shown that ${\rm Rate}(M)$ decreases first then increases, and there exists a minimum point of ${\rm Rate}(M)$. For $M$ that is too small, the rate will be larger than 1, which means the hybrid iterative method will diverge. The rate increases and tends to the convergence rate of Richardson as $M\to\infty$. 
\begin{figure}[htbp]
    \centering
    \includegraphics[width=0.6\textwidth]{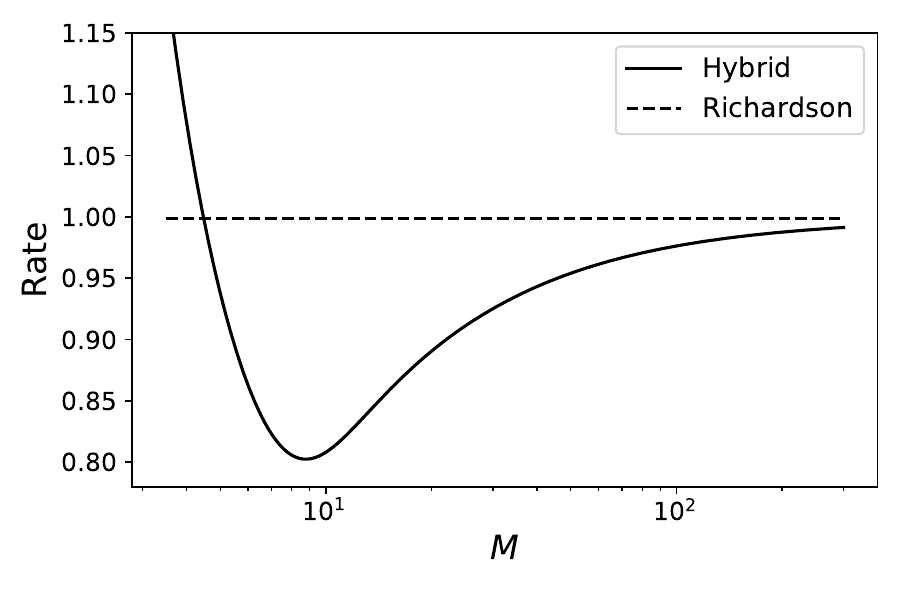}
    \caption{\textbf{An example of ${\rm Rate}(M)$.} In this case $\eta_1=0.999, \eta_2=0.5, \epsilon=0.1, R=10$.}
    \label{fig:rate_M}
\end{figure}

\subsection{Discussion for complicated cases}
At this point, we have systematically analyzed the hybrid iterative method for the one-dimensional case with Richardson iteration (or equivalently damped Jacobi iteration). For more complicated case, such as the high-dimensional problem with Gauss-Seidel iteration which is a better smoother, we have to apply other effective tools \cite{brandt1977multi,hackbusch2013multi}. The smoothing property of GS is conveniently analyzed by using local mode analysis introduced by Brandt \cite{brandt1977multi}. Here let us show the idea and promote the discussion on the hybrid Gauss-Seidel iteration.

Recall that the 2-d Poisson equation with the uniform grid and the lexicographical order can be discretized as
\begin{equation}\label{eq:ideal_problem}
4\mu_{i,j}-(\mu_{i-1,j}+\mu_{i,j-1}+\mu_{i,j+1}+\mu_{i+1,j})=b_{i,j}.
\end{equation}
To begin with, we add some idealized assumptions and simplifications as in the local mode analysis:
\begin{itemize}
\item[1.] Assume that boundary conditions are neglected and the problem is defined on an infinite grid with a period of $n+1$, i.e.
\begin{equation}
\mu_{i+k\cdot(n+1),j+l\cdot(n+1)}=\mu_{i,j},\quad b_{i+k\cdot(n+1),j+l\cdot(n+1)}=b_{i,j},\quad \forall i,j,k,l\in\mathbb{Z}.
\end{equation}
\item[2.] Assume that the Gauss-Seidel iteration is implicitly defined as
\begin{equation}\label{eq:GS_implicit}
4\tilde{\mu}_{i,j}-(\tilde{\mu}_{i-1,j}+\tilde{\mu}_{i,j-1}+\bar{\mu}_{i,j+1}+\bar{\mu}_{i+1,j})=b_{i,j},\quad \forall i,j\in\mathbb{Z},
\end{equation}
where $\tilde{\mu}$ is updated from $\bar{\mu}$. Note that Eq. (\ref{eq:GS_implicit}) leads to an algebraic system for $\tilde{\mu}$, which is uniquely solvable since it is strictly diagonally dominant.
\end{itemize}
Denote $\tilde{e}_{i,j}=\mu_{i,j}-\tilde{\mu}_{i,j}$ and $\bar{e}_{i,j}=\mu_{i,j}-\bar{\mu}_{i,j}$, then
\begin{equation}\label{eq:GS_update}
4\tilde{e}_{i,j}-(\tilde{e}_{i-1,j}+\tilde{e}_{i,j-1}+\bar{e}_{i,j+1}+\bar{e}_{i+1,j})=0.
\end{equation}
We expand these errors by the discrete Fourier transformation, as
\begin{equation}\label{eq:discrete_fourier}
\tilde{e}_{i,j}=\sum_{\theta\in\Theta}\tilde{c}(\theta)\Psi^{\theta}_{i,j},\quad \bar{e}_{i,j}=\sum_{\theta\in\Theta}\bar{c}(\theta)\Psi^{\theta}_{i,j},\quad 0\leq i,j\leq n,
\end{equation}
where
\begin{equation}
\Psi^{\theta}_{i,j}=\frac{1}{n+1}\mathbf{e}^{\mathbf{i}(i\theta_1+j\theta_2)},\quad \mathbf{i}^2=-1,\quad \theta=(\theta_1, \theta_2),
\end{equation}
and
\begin{equation}
\Theta=\left\{\frac{2\pi}{n+1}(k,l)\bigg|-\frac{n}{2}\leq k,l\leq\frac{n+1}{2},\quad k,l\in \mathbb{Z}\right\}.
\end{equation}
Substituting (\ref{eq:discrete_fourier}) into (\ref{eq:GS_update}), we obtain
\begin{equation}
\zeta(\theta):=\left|\frac{\tilde{c}(\theta)}{\bar{c}(\theta)}\right|=\left|\frac{\mathbf{e}^{\mathbf{i}\theta_1}+\mathbf{e}^{\mathbf{i}\theta_2}}{4-\mathbf{e}^{\mathbf{i}\theta_1}-\mathbf{e}^{\mathbf{i}\theta_2}}\right|.
\end{equation}
The smoothing factor is defined as
\begin{equation}
\zeta_{\rho}=\max_{\rho\pi\leq|\theta|\leq\pi}\zeta(\theta),\quad |\theta|:=\max(|\theta_1|,|\theta_2|),\quad 0<\rho<1.
\end{equation}
It can be shown that $\zeta_{1/2}=\zeta(\frac{\pi}{2},\arccos(\frac{4}{5}))=\frac{1}{2}$, which means the GS iteration can efficiently eliminate the high-frequency errors. 

It is worth noting that $\zeta(0,0)=1$, which indicates that the lowest-frequency component will never be eliminated. In fact, the problem (\ref{eq:ideal_problem}) does not determine a solution, since a change of a constant on $\mu$ does not affect the relationship. For convenience, we only consider the error excluding the lowest-frequency component in the following.

Based on this framework, we proceed the discussion on the network correction step. Assume that the MIONet-based iterative step $\bar{\mathcal{M}}$ sufficiently learns the low-frequency modes, i.e.,
\begin{equation}\label{eq:mionet_learn}
4[\bar{\mathcal{M}}\Psi^{\theta}]_{i,j}-([\bar{\mathcal{M}}\Psi^{\theta}]_{i-1,j}+[\bar{\mathcal{M}}\Psi^{\theta}]_{i,j-1}+[\bar{\mathcal{M}}\Psi^{\theta}]_{i,j+1}+[\bar{\mathcal{M}}\Psi^{\theta}]_{i+1,j})\approx\Psi^\theta_{i,j},\quad 0<|\theta|\leq\rho\pi,
\end{equation}
where $0<\rho<1$ is a constant that determines the (low-)frequency components learned. Denote
\begin{equation}\label{eq:mode_express}
\bar{\mathcal{M}}\Psi^{\theta}=\sum_{\varphi\in\Theta}\alpha_{\theta}(\varphi)\cdot \Psi^\varphi,
\end{equation}
we substitute (\ref{eq:mode_express}) into (\ref{eq:mionet_learn}) and similarly derive that
\begin{equation}
\sum_{\varphi\in\Theta}(4-2\cos(\varphi_1)-2\cos(\varphi_2))\alpha_{\theta}(\varphi)\Psi^{\varphi}\approx\Psi^{\theta},\quad 0<|\theta|\leq\rho\pi,
\end{equation}
which leads to
\begin{equation}\label{eq:dirac}
(4-2\cos(\varphi_1)-2\cos(\varphi_2))\alpha_{\theta}(\varphi)\approx\delta_{\theta}^{\varphi}\quad ({\rm Dirac\ delta\ function}),\quad 0<|\theta|\leq\rho\pi.
\end{equation}
Denote
\begin{equation}
\gamma_\theta^\varphi:=\delta_{\theta}^{\varphi}-(4-2\cos(\varphi_1)-2\cos(\varphi_2))\alpha_{\theta}(\varphi),
\end{equation}
then we suppose that $|\gamma_\theta^\varphi|$ is as small as needed for $0<|\theta|\leq\rho\pi$. 

Now we consider the MIONet iteration step, i.e.,
\begin{equation}
\tilde{\mu}=\bar{\mu}+\bar{\mathcal{M}}(\bar{r}),
\end{equation}
or equivalently
\begin{equation}
\tilde{e}=\bar{e}-\bar{\mathcal{M}}(\bar{r}),
\end{equation}
where $\tilde{e}=\mu-\tilde{\mu}$, $\bar{e}=\mu-\bar{\mu}$, and
\begin{equation}
\begin{split}
\bar{r}_{i,j}=&b_{i,j}-(4\bar{\mu}_{i,j}-(\bar{\mu}_{i-1,j}+\bar{\mu}_{i,j-1}+\bar{\mu}_{i,j+1}+\bar{\mu}_{i+1,j})) \\
=&4\mu_{i,j}-(\mu_{i-1,j}+\mu_{i,j-1}+\mu_{i,j+1}+\mu_{i+1,j})\\
&-(4\bar{\mu}_{i,j}-(\bar{\mu}_{i-1,j}+\bar{\mu}_{i,j-1}+\bar{\mu}_{i,j+1}+\bar{\mu}_{i+1,j})) \\
=&4\bar{e}_{i,j}-(\bar{e}_{i-1,j}+\bar{e}_{i,j-1}+\bar{e}_{i,j+1}+\bar{e}_{i+1,j}).
\end{split}
\end{equation}
We expand $\tilde{e},\bar{e},\bar{r}$ as before, i.e.,
\begin{equation}
\tilde{e}=\sum_{\theta\in\Theta}\tilde{c}(\theta)\Psi^{\theta},\quad\bar{e}=\sum_{\theta\in\Theta}\bar{c}(\theta)\Psi^{\theta},\quad \bar{r}=\sum_{\theta\in\Theta}(4-2\cos(\theta_1)-2\cos(\theta_2))\bar{c}(\theta)\Psi^{\theta},
\end{equation}
then we have
\begin{equation}
\begin{split}
\sum_{\theta\in\Theta}\tilde{c}(\theta)\Psi^{\theta}=&\sum_{\theta\in\Theta}\bar{c}(\theta)\Psi^{\theta}-\bar{\mathcal{M}}\left(\sum_{\theta\in\Theta}(4-2\cos(\theta_1)-2\cos(\theta_2))\bar{c}(\theta)\Psi^{\theta}\right) \\
=&\sum_{\theta\in\Theta}\bar{c}(\theta)\Psi^{\theta}-\sum_{\theta\in\Theta}(4-2\cos(\theta_1)-2\cos(\theta_2))\bar{c}(\theta)\sum_{\varphi\in\Theta}\alpha_{\theta}(\varphi)\cdot \Psi^\varphi \\
=&\sum_{\theta\in\Theta}\bar{c}(\theta)\Psi^{\theta}-\sum_{\theta\in\Theta}\sum_{\varphi\in\Theta}(4-2\cos(\varphi_1)-2\cos(\varphi_2))\bar{c}(\varphi)\alpha_{\varphi}(\theta)\cdot \Psi^\theta.
\end{split}
\end{equation}
Consequently,
\begin{equation}\label{eq:c_theta_update}
\begin{split}
\tilde{c}(\theta)=&\bar{c}(\theta)-\sum_{\varphi\in\Theta}\bar{c}(\varphi)(4-2\cos(\varphi_1)-2\cos(\varphi_2))\alpha_{\varphi}(\theta) \\
=&\sum_{\varphi\in\Theta}\bar{c}(\varphi)\cdot\frac{4-2\cos(\varphi_1)-2\cos(\varphi_2)}{4-2\cos(\theta_1)-2\cos(\theta_2)}\cdot\gamma^\theta_\varphi,\quad |\theta|>0.
\end{split}
\end{equation}
Return to the algorithm of the hybrid iterative method, the network correction step is performed following $M-1$ GS iteration steps, so that $\bar{c}(\varphi)$ can be written as
\begin{equation}
\bar{c}(\varphi)=\bar{c}^{(M-1)}(\varphi),
\end{equation} 
then we have the estimate
\begin{equation}\label{eq:estimate_zeta_rho}
|\bar{c}(\varphi)|=|\bar{c}^{(M-1)}(\varphi)|\leq(\zeta_{\rho})^{M-1}|\bar{c}^{(0)}(\varphi)|,\quad |\varphi|\geq\rho\pi,
\end{equation}
and the low-frequency component $|\bar{c}^{(M-1)}(\varphi)|$ ($0<|\varphi|<\rho\pi$) has a smoothing factor $\zeta_0$ greater than $\zeta_{\rho}$, i.e.
\begin{equation}\label{eq:estimate_zeta_0}
|\bar{c}^{(M-1)}(\varphi)|\leq(\zeta_0)^{M-1}|\bar{c}^{(0)}(\varphi)|,\quad \zeta_\rho<\zeta_0<1,\quad 0<|\varphi|<\rho\pi.
\end{equation}
Now denote
\begin{equation}
e^{(M)}:=\sum_{\varphi\in\Theta,|\varphi|>0}\tilde{c}(\varphi)\Psi^{\varphi},\quad e^{(0)}:=\sum_{\varphi\in\Theta,|\varphi|>0}\bar{c}^{(0)}(\varphi)\Psi^{\varphi},
\end{equation}
then
\begin{equation}
\begin{split}
\norm{e^{(M)}}_2=&\sqrt{\sum_{|\theta|>0}|\tilde{c}(\theta)|^2} \\
\leq&\sum_{|\theta|>0}|\tilde{c}(\theta)| \\
\leq&\sum_{|\theta|>0}\sum_{|\varphi|>0}|\bar{c}(\varphi)|\cdot\left|\frac{4-2\cos(\varphi_1)-2\cos(\varphi_2)}{4-2\cos(\theta_1)-2\cos(\theta_2)}\right|\cdot|\gamma^\theta_\varphi|,\\
\end{split}
\end{equation}
by (\ref{eq:c_theta_update}). With (\ref{eq:estimate_zeta_rho}) and (\ref{eq:estimate_zeta_0}), we further derive that
\begin{equation}
\begin{split}
\norm{e^{(M)}}_2\leq&\sum_{|\varphi|>0}|\bar{c}^{(M-1)}(\varphi)|\cdot\sum_{|\theta|>0}\left|\frac{4-2\cos(\varphi_1)-2\cos(\varphi_2)}{4-2\cos(\theta_1)-2\cos(\theta_2)}\right|\cdot|\gamma^\theta_\varphi| \\
\leq&\sum_{0<|\varphi|<\rho\pi}(\zeta_0)^{M-1}|\bar{c}^{(0)}(\varphi)|\cdot\sum_{|\theta|>0}\left|\frac{4-2\cos(\varphi_1)-2\cos(\varphi_2)}{4-2\cos(\theta_1)-2\cos(\theta_2)}\right|\cdot|\gamma^\theta_\varphi| \\
&+\sum_{|\varphi|\geq\rho\pi}(\zeta_{\rho})^{M-1}|\bar{c}^{(0)}(\varphi)|\cdot\sum_{|\theta|>0}\left|\frac{4-2\cos(\varphi_1)-2\cos(\varphi_2)}{4-2\cos(\theta_1)-2\cos(\theta_2)}\right|\cdot|\gamma^\theta_\varphi|\\
\leq&\left((\zeta_0)^{M-1}\cdot\epsilon+(\zeta_{\rho})^{M-1}\cdot R\right)\norm{e^{(0)}}_2,
\end{split}
\end{equation}
where
\begin{equation}
\begin{cases}
\epsilon=\sum_{0<|\varphi|<\rho\pi,|\theta|>0}\left|\frac{4-2\cos(\varphi_1)-2\cos(\varphi_2)}{4-2\cos(\theta_1)-2\cos(\theta_2)}\right|\cdot|\gamma^\theta_\varphi|,\\
R=\sum_{|\varphi|\geq\rho\pi,|\theta|>0}\left|\frac{4-2\cos(\varphi_1)-2\cos(\varphi_2)}{4-2\cos(\theta_1)-2\cos(\theta_2)}\right|\cdot|\gamma^\theta_\varphi|,
\end{cases}
\end{equation}
and $\epsilon$ is supposed to be small as discussed before. The relationship
\begin{equation}
\norm{e^{(M)}}_2\leq\left((\zeta_0)^{M-1}\cdot\epsilon+(\zeta_{\rho})^{M-1}\cdot R\right)\norm{e^{(0)}}_2
\end{equation}
immediately leads to the final upper bound of the average convergence rate
\begin{equation}
{\rm Rate}(M)=\zeta_0\cdot\left(\frac{\epsilon}{\zeta_0}+\frac{R}{\zeta_\rho}\cdot\left(\frac{\zeta_{\rho}}{\zeta_0}\right)^{M-1}\right)^{\frac{1}{M}}.
\end{equation}

According to the above deduction, we have studied the convergence rate of this hybrid iterative method for 2-d problem with GS iteration, which is consistent with Eq. (\ref{eq:convergence_rate}). It is expected to promote the discussion and obtain further theoretical results for more complicated cases. We will keep it for future research due to space limitations.


\section{Numerical experiments}\label{sec:experiments}
The experiments are implemented in \verb!c++! with \textbf{Eigen} library for sparse matrices and basic iterations, and \textbf{LibTorch} library for model inference. The used models are pre-trained via \textbf{PyTorch}. The running devices are \textbf{Intel(R) Core(TM) i7-10750H} for CPU and \textbf{NVIDIA GeForce RTX 2070} for GPU.
\subsection{One-dimensional Poisson equation}
Consider the 1-d Poisson equation
\begin{equation}\label{eq:1d_poisson}
	\begin{cases}
		-\nabla\cdot(k\nabla u)=f\quad \mbox{in} \ (0,1),
		\\
		u(0)=u(1)=0.
	\end{cases}	
\end{equation}
To firstly construct the dataset, we generate $k$ and $f$ by Gaussian Process (GP) with RBF kernel. $k$ is generated by GP with mean 1, std 0.2, length scale 0.1, while $f$ is with mean 0, std 1, length scale 0.1. We totally generate 5000 data points as training set $\T=\{(k_i,f_i),u_i\}_{i=1}^{5000}$. After training an MIONet on the dataset $\T$, we obtain an MIONet-based iteration step $\bar{\mathcal{M}}:\R^n\to\R^n$ (linear), where $n=48$ is the number of the interior nodal points in the uniform grid for the interval $(0,1)$. Here the size of MIONet is $[50, 100, 100, 100]$ for $k$'s branch net, $[50, 100]$ for $f$'s branch net, and $[1,100,100,100]$ for the trunk net. Note that we remove the bias in $f$'s branch net and also the bias $b$ in Eq. (\ref{eq:MIONet_low}) to guarantee the linearity w.r.t. $f$. Figure \ref{fig:1d_poisson_prediction} shows an example of a couple of inputs $k$ and $f$ as well as the corresponding solution $u$ and the MIONet prediction.

\begin{figure}[htbp]
    \centering
    \includegraphics[width=1.0\textwidth]{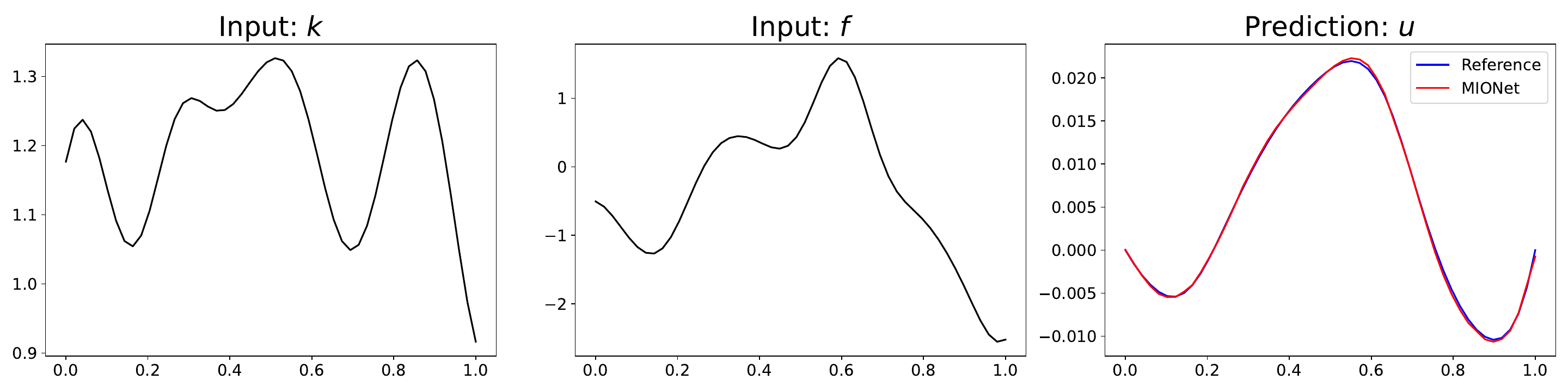}
    \caption{\textbf{An example of one prediction of MIONet for the 1-d Poisson equation.} }
    \label{fig:1d_poisson_prediction}
\end{figure}

We then apply the $P_1$ element together with the Richardson iteration (set $\omega=\frac{h}{4}=\frac{1}{4\times 49}$) to solve this 1-d Poisson equation for $k=1$. The MIONet iteration $\bar{\mathcal{M}}$ is performed every $M=20$ steps to correct low-frequency errors in the Richardson-MIONet method (we also show the convergence rate versus $M$ in Figure \ref{fig:Poisson_1d_versusM}, which is consistent with the theoretical results as Eq. (\ref{eq:convergence_rate}) and Figure \ref{fig:rate_M}). The error threshold for stopping iteration is $1\times 10^{-14}$. We list the numbers of iterations of the Richardson and the Richardson-MIONet in Table \ref{tab:1d_poisson}. The number of iterations of the original Richardson method is 157 times that of the Richardson-MIONet method. Since the 1-d case converges too fast, we leave the study of the convergence time to the later 2-d experiment.

\begin{table}[htbp]
\centering
\begin{tabular}{|c|c|c|}
\hline
          & $\#$Iterations  & Speed up\\ \hline
Richardson        & 28396       &      /    \\ \hline
Richardson-MIONet & 181        & $\times$157      \\ \hline
\end{tabular}
\caption{\textbf{Comparison of the Richardson iterative method and the Richardson-MIONet hybrid iterative method for the 1-d Poisson equation.} The number of iterations of the original Richardson method is 157 times that of the Richardson-MIONet method.}
\label{tab:1d_poisson}
\end{table}

\begin{figure}[htbp]
    \centering
    \includegraphics[width=0.6\textwidth]{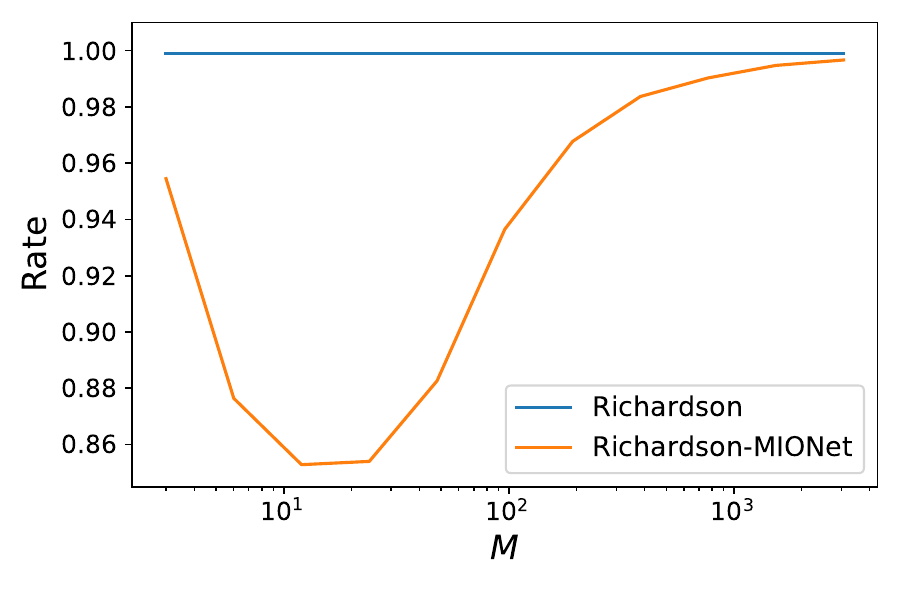}
    \caption{\textbf{Convergence rate of the Richardson-MIONet method versus the correction period $M$.} The tendency is consistent with the theoretical results as Eq. (\ref{eq:convergence_rate}) and Figure \ref{fig:rate_M}.}
    \label{fig:Poisson_1d_versusM}
\end{figure}

We further observe the spectral error of the iterative methods, and the results are shown in Figure \ref{fig:1d_poisson_spectrum}. The high-frequency errors are quickly eliminated by the Richardson iteration in both two methods. However, the low-frequency errors are difficult to eliminate via the original Richardson method. In each correction step, MIONet efficiently modifies the low-frequency errors. Although MIONet brings new high-frequency errors, they will be quickly clear up soon. After about 180 steps, the Richardson-MIONet method has already achieved a machine accuracy, while the original Richardson method still keeps a significant error.

\begin{figure}[htbp]
    \centering
    \includegraphics[width=1.0\textwidth]{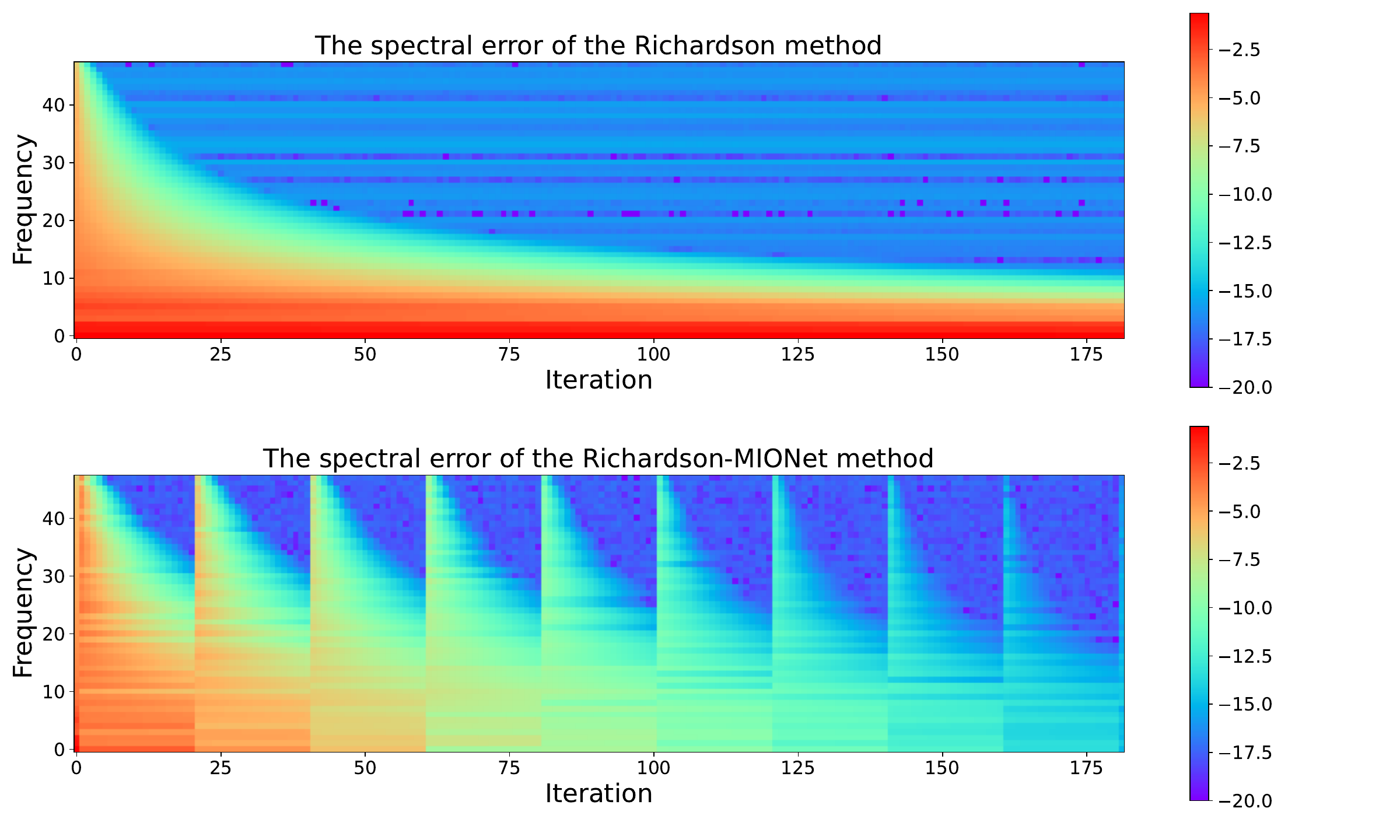}
    \caption{\textbf{The spectral error of the iterative methods.} The numerical value of each pixel is $\log_{10}|{\rm spectral\ error}|$. The high-frequency errors are quickly eliminated by the Richardson iteration in both two methods. However, the low-frequency errors are difficult to eliminate via the original Richardson method. MIONet efficiently modifies the low-frequency errors every $M=20$ steps. Although MIONet brings new high-frequency errors, they will be quickly clear up soon. After about 180 steps, the Richardson-MIONet method has already achieved a machine accuracy, while the original Richardson method still keeps a significant error.}
    \label{fig:1d_poisson_spectrum}
\end{figure}

\subsection{Two-dimensional Poisson equation}
Consider the 2-d Poisson equation
\begin{equation}\label{eq:2d_poisson}
	\begin{cases}
		-\nabla\cdot(k\nabla u)=f &\quad \mbox{in} \; \Omega,
		\\
		u = 0 &\quad \mbox{on} \; \partial\Omega,
	\end{cases}	
\end{equation}
where $\Omega=(0,1)^2$. We also generate $k$ and $f$ as the same setting as the 1-d case. The difference is that we set the length scale to 0.2. We totally generate 5000 data points as training set $\T=\{(k_i,f_i),u_i\}_{i=1}^{5000}$. After training an MIONet on the dataset $\T$, we obtain an approximate solution operator. Here the size of MIONet is $[100^2, 500, 500, 500]$ for $k$'s branch net, $[100^2, 500]$ for $f$'s branch net, and $[2,500,500,500]$ for the trunk net. We also remove the corresponding biases as in the 1-d case. An example of a couple of inputs $k$ and $f$ as well as the corresponding solution $u$ and the MIONet prediction is shown in Figure \ref{fig:2d_poisson}.

\begin{figure}[htbp]
    \centering
    \includegraphics[width=1.0\textwidth]{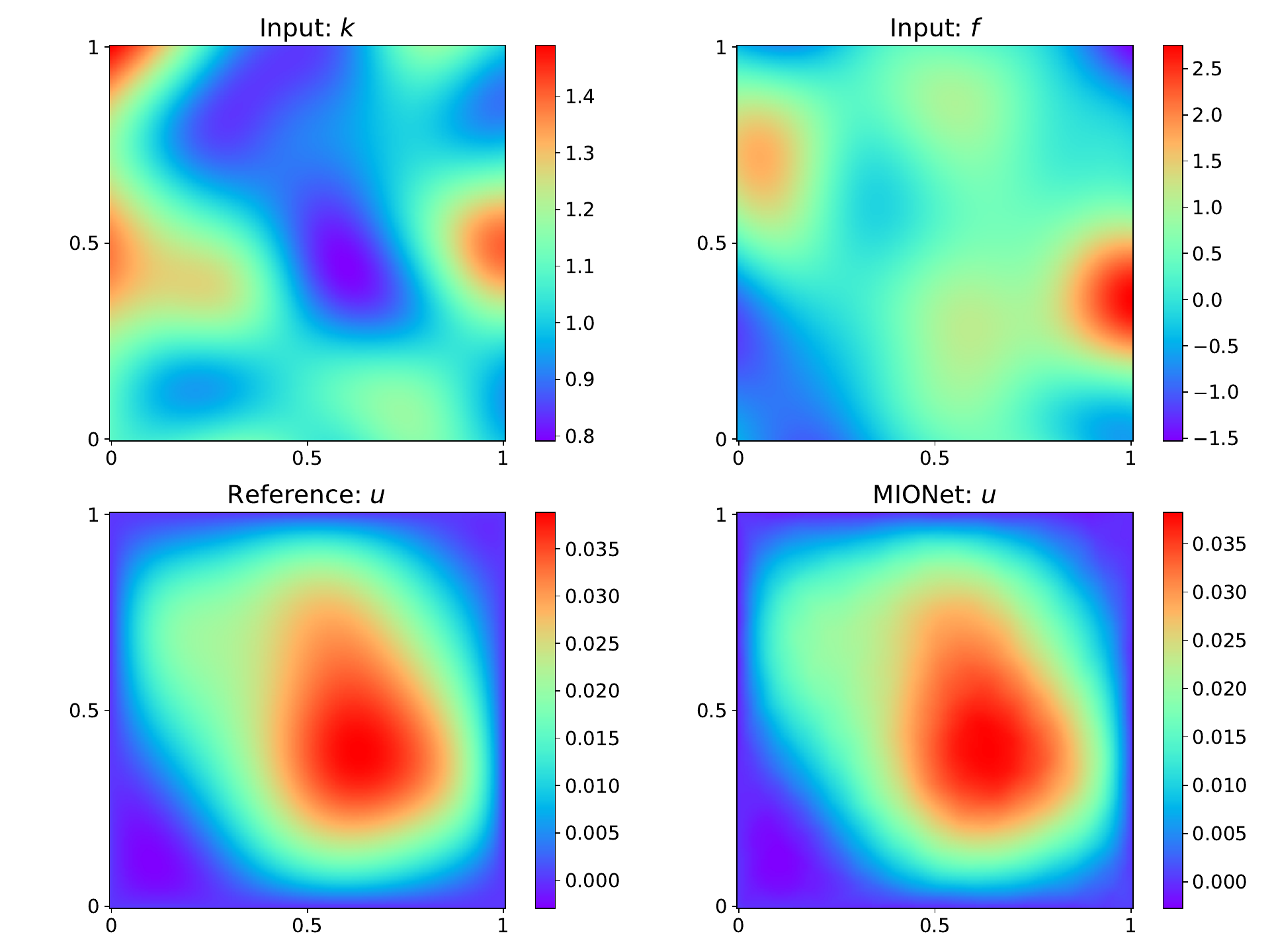}
    \caption{\textbf{An example of one prediction of MIONet for the 2-d Poisson equation.} }
    \label{fig:2d_poisson}
\end{figure}

We then apply the $P_1$ element together with the Gauss-Seidel (GS) iteration to solve this 2-d Poisson equation with grid size $1025\times 1025$. The MIONet iteration $\bar{\mathcal{M}}:\R^n\to\R^n$ ($n=1023\times1023$) is performed every $M$ steps to correct low-frequency errors in the GS-MIONet method. The error threshold for stopping iteration is $1\times 10^{-12}$. We list the numbers of convergence iterations and the convergence time of GS-MIONet given different $M$ in Table \ref{tab:2d_poisson_1024}, \ref{tab:2d_poisson_1024_2}. We also show the trends of the convergence steps and the convergence time versus $M$ in Figure \ref{fig:Poisson_2d_versus_M}. The best speed of the GS-MIONet method is nearly 30 times that of the original GS method at $M=8000$, while the number of convergence iterations reaches its minimum at $M=4000$. This gap exists since the inference time of MIONet accounts for a significant portion in the hybrid iteration. When $M>2556893$ (i.e. the number of convergence iterations of the GS method), MIONet only gives an initial guess during the hybrid iteration, which provides a $22\%$ acceleration.

\begin{table}[htbp]
\centering
\begin{tabular}{|c|c|c|c|c|c|c|c|c|c|}
\hline
$M$          & GS & 500 & 1000 & 2000   & 4000 & 8000 & 16000 & 32000 & 64000  \\ \hline
$\#$Iterations & 2556893 &  div. & 167001 & 94001 & 78471 & 82484 & 96001 & 160001 & 256001    \\ \hline
Time (s)   & 27103 &  div. & 2641 & 1222 & 923 & 919 & 1049 & 1722 & 2750    \\ \hline
Speed up   & /  &  div. & $\times$10.3  & $\times$22.2  & $\times$29.4  & $\times$29.5    & $\times$25.8    & $\times$15.7    & $\times$9.9       \\ \hline
\end{tabular}
\caption{\textbf{The number of convergence iterations and the convergence time of GS-MIONet hybrid iterative method for the 2-d Poisson equation versus the correction period $M$ (part 1).} The size of grid is $1025\times1025$. ``div.'' means the iteration diverges under the given $M$.}
\label{tab:2d_poisson_1024}
\end{table}

\begin{table}[htbp]
\centering
\begin{tabular}{|c|c|c|c|c|c|c|c|}
\hline
$M$         & 128000 & 256000 & 512000 & 1024000 & 2048000 & 4096000 & 8192000 \\ \hline
$\#$Iterations & 384640 & 768001 & 1024001 & 1439557 & 2048001 & 2071460 & 2071460  \\ \hline
Time (s)   & 4079 & 8173 & 10804 & 15467 & 22058 & 22181 & 22181  \\ \hline
Speed up  & $\times$6.6   & $\times$3.3   & $\times$2.5  & $\times$1.75   & $\times$1.23  & $\times$1.22   & $\times$1.22 \\ \hline
\end{tabular}
\caption{\textbf{The number of convergence iterations and the convergence time of GS-MIONet hybrid iterative method for the 2-d Poisson equation versus the correction period $M$ (part 2).} The size of grid is $1025\times1025$. When $M>2556893$ (i.e. the number of convergence iterations of the GS method), MIONet only gives an initial guess during the hybrid iteration.}
\label{tab:2d_poisson_1024_2}
\end{table}

\begin{figure}[htbp]
    \centering
    \includegraphics[width=1.0\textwidth]{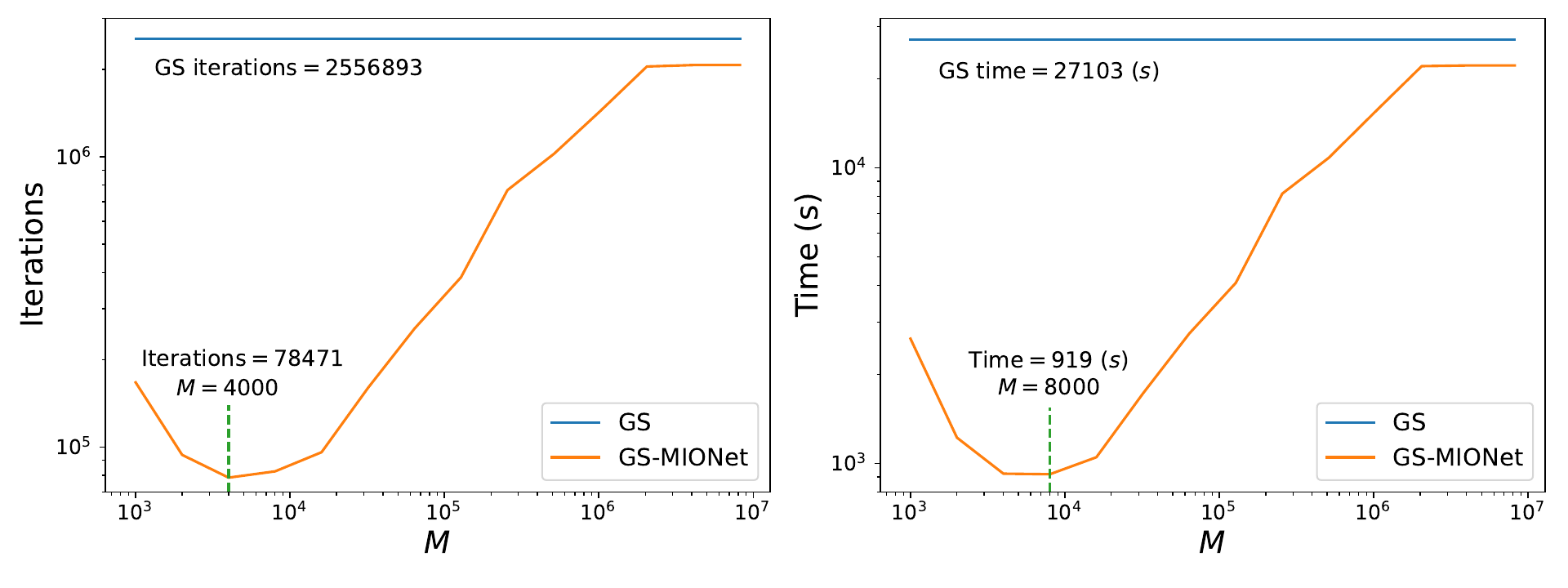}
    \caption{\textbf{Performance of GS-MIONet versus the correction period $M$.} We use the MIONet trained on $100\times 100$ grid, and to perform hybrid iteration on $1025\times 1025$ grid. The number of convergence iterations reaches its minimum at $M=4000$, while the convergence time reaches its minimum at $M=8000$. The best convergence time of GS-MIONet method is 919 seconds, the speed of which is nearly 30 times that of the original GS method. When $M$ is large enough, the MIONet only gives an initial guess during the hybrid iteration, which provides a $22\%$ acceleration.}
    \label{fig:Poisson_2d_versus_M}
\end{figure}

Next we compare the hybrid iterative method to the multigrid method. We first solve this equation by different levels of multigrid methods, then for each case, we again test the corresponding multigrid-MIONet hybrid iterative method. In this hybrid method, we regard a V-cycle of the multigrid method as one step, and insert one network correction step every $M$ V-cycles. The results are recorded in Table \ref{tab:2d_poisson_mg} and visually presented in Figure \ref{fig:2d_poisson_mg}. We observe that the proposed hybrid iterative method can flexibly utilize the advanced multigrid method. The speed of GS-MIONet is between the multigrid methods of level 3 and 4. There is an acceleration effect all the way up to level 6. For a high level ($\geq 7$) multigrid method, it converges so fast that the inference time of MIONet is larger than the multigrid's convergence time. Unsurprisingly in such a case the hybrid method converges slower compared to the original multigrid method. However, in practice it is costly to construct high level multigrid, while the use of the hybrid iterative method has no extra burden for users as long as the pre-trained model is provided.

\begin{table}[htbp]
\centering
\begin{tabular}{|c|c|c|c|c|c|c|c|c|c|c|}
\hline
Method     & GS & Mg2 & Mg3 & Mg4 & Mg5 & Mg6 & Mg7 & Mg8 & Mg9 & Mg10 \\ \hline
Basic time (s)  & 27103 & 8880 & 2210 & 567 & 139 & 35 & 8.7 & 2.4 & 1.7 & 1.7  \\ \hline
Hybrid time (s)  & 919 & 337 & 118 & 56 & 34 & 22 & 12.2 & 7.0 & 6.4 & 6.4 \\ \hline
Speed up  & $\times$29.5 & $\times$26.4 & $\times$18.7 & $\times$10.1 & $\times$4.1 & $\times$1.6 & / & / & / & / \\ \hline
\end{tabular}
\caption{\textbf{Multigrid-MIONet method.} The time of multigrid methods and corresponding hybrid methods on $1025\times1025$ grid. The correction period $M$ is tuned for each hybrid case.}
\label{tab:2d_poisson_mg}
\end{table}

\begin{figure}[htbp]
    \centering
    \includegraphics[width=1.0\textwidth]{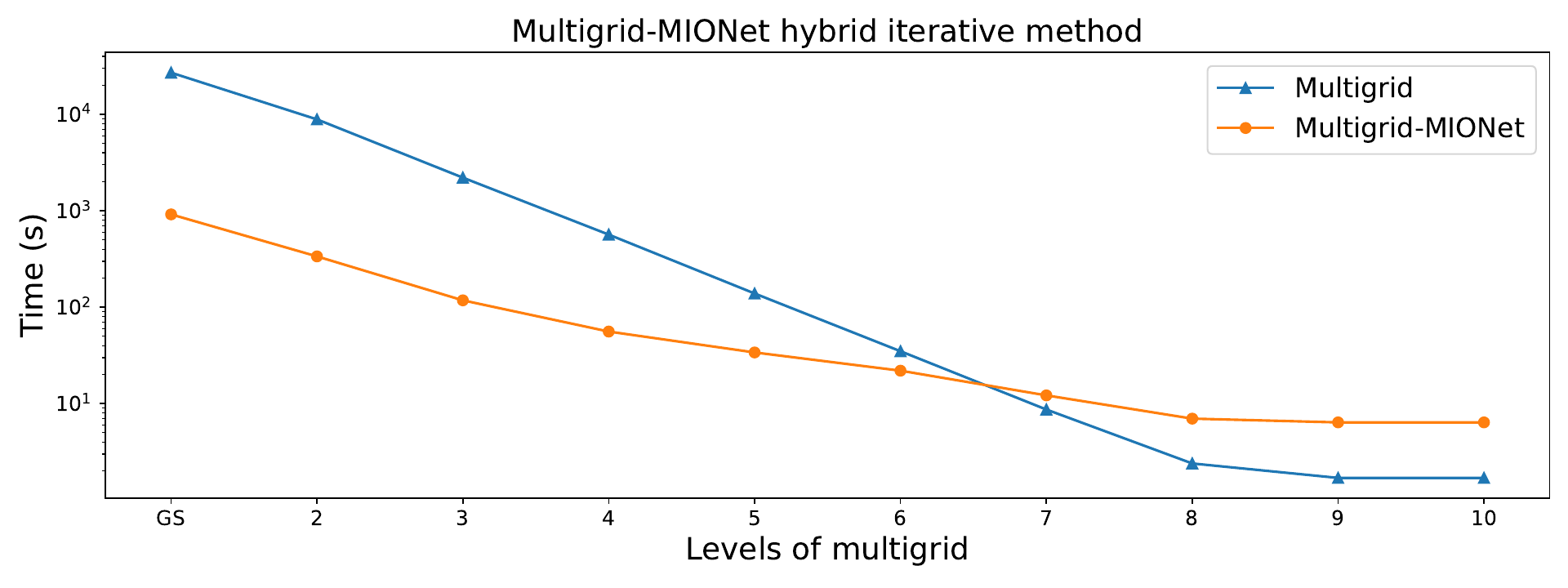}
    \caption{\textbf{Multigrid-MIONet method for the 2-d Poisson equation on $1025\times1025$ grid.} There is an acceleration effect all the way up to level 6, and the speed of GS-MIONet is between the multigrid methods of level 3 and 4.}
    \label{fig:2d_poisson_mg}
\end{figure}

\subsection{Inhomogeneous boundary condition}
Finally we consider the 2-d Poisson equation with inhomogeneous boundary condition:
\begin{equation}
	\begin{cases}
		-\nabla\cdot(k\nabla u)=f &\quad \mbox{in} \; \Omega,
		\\
		u = g &\quad \mbox{on} \; \partial\Omega,
	\end{cases}	
\end{equation}
where $\Omega=(0,1)^2$, and the boundary condition $g$ is also changing. We generate $k$ and $f$ as the same setting as the previous 2-d case. The $g$ is regarded as a function with a period of 4 (anticlockwise along the boundary), which is generated by GP with exponential sine squared kernel, where we set mean, std, length scale and period to 0, 0.05, 1 and 4, respectively. We totally generate 5000 data points as training set $\T=\{(k_i,(f_i, g_i)),u_i\}_{i=1}^{5000}$. After training an MIONet on the dataset $\T$, we obtain an approximate solution operator. Here the size of MIONet is $[100^2, 500, 500, 500]$ for $k$'s branch net, $[100^2, 500]$ for $(f,g)$'s branch net, and $[2,500,500,500]$ for the trunk net. We remove the corresponding biases as before. An example of inputs $k$, $f$ and $g$ as well as the corresponding solution $u$ and the MIONet prediction is shown in Figure \ref{fig:2d_poisson_boundary}.

\begin{figure}[htbp]
    \centering
    \includegraphics[width=1.0\textwidth]{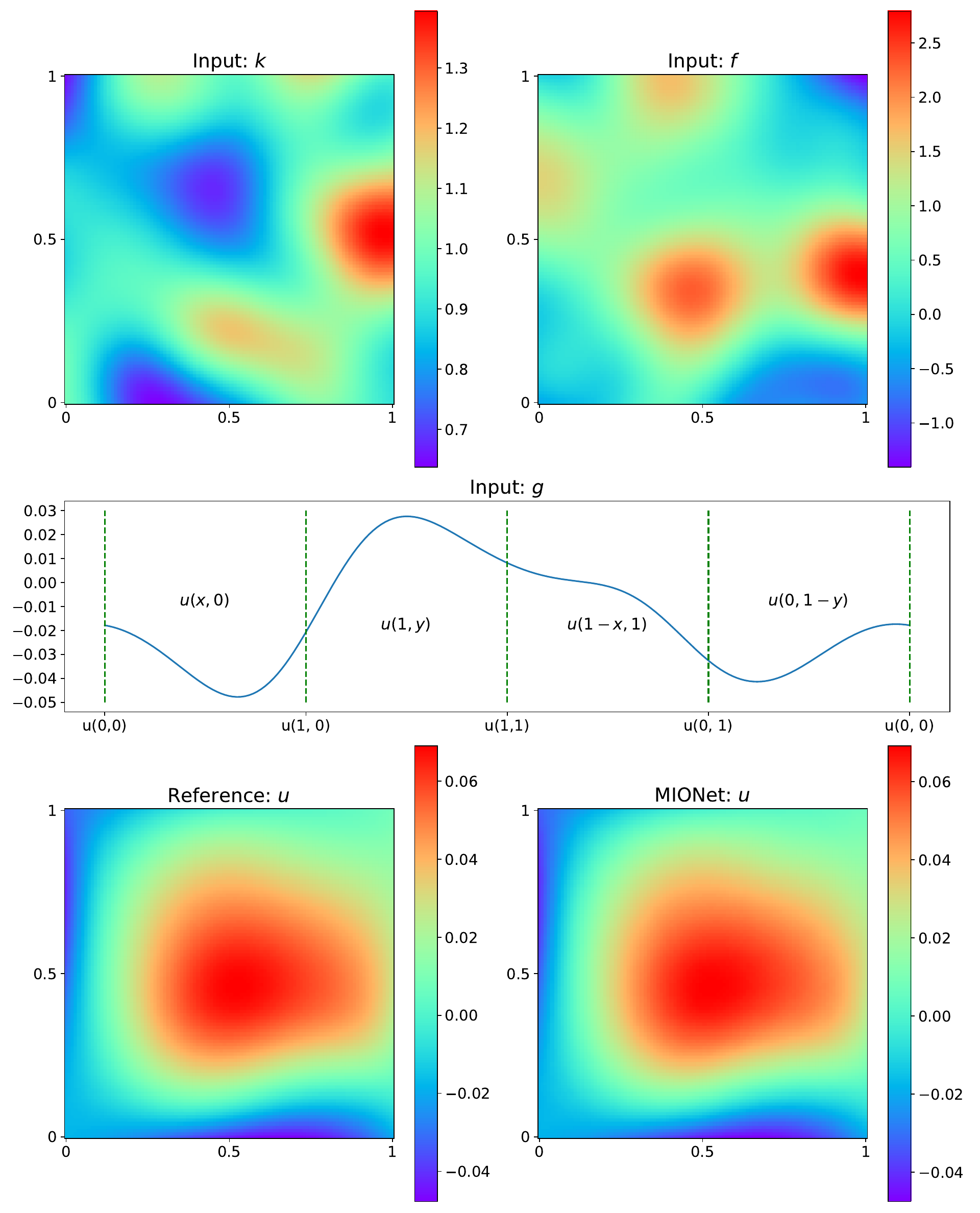}
    \caption{\textbf{An example of one prediction of MIONet for the 2-d Poisson equation with inhomogeneous boundary condition.} }
    \label{fig:2d_poisson_boundary}
\end{figure}

We then apply the $P_1$ element together with the GS iteration to solve this equation on the $500\times 500$ uniform grid. The MIONet iteration $\bar{\mathcal{M}}:\R^n\to\R^n$ ($n=500\times500$) is performed every $M=1600$ steps to correct low-frequency errors in the GS-MIONet method. The error threshold for stopping iteration is $1\times 10^{-12}$. We show the number of convergence iterations and the convergence time of the GS(-MIONet) method in Table \ref{tab:2d_poisson_boundary}. The speed of the GS-MIONet method is nearly 24 times that of the original GS method.

\begin{table}[htbp]
\centering
\begin{tabular}{|c|c|c|c|}
\hline
          & $\#$Iterations & Time (s)  & Speed up\\ \hline
GS        & 672230   &   1652  &      /    \\ \hline
GS-MIONet & 20963  &  69  & $\times$23.9      \\ \hline
\end{tabular}
\caption{\textbf{Comparison of the GS iterative method and the GS-MIONet hybrid iterative method for the 2-d Poisson equation with inhomogeneous boundary condition.} The size of grid is $500\times500$. The speed of the GS-MIONet method is 23.9 times that of the original GS method.}
\label{tab:2d_poisson_boundary}
\end{table}

\section{Conclusions}\label{sec:conclusions}
We have proposed a hybrid iterative method based on MIONet for PDEs, which combines the traditional numerical iterative solver and the neural operator, and further systematically analyzed its theoretical properties, including the convergence condition, the spectral behavior, as well as the convergence rate, in terms of the errors of the discretization and the model inference. We have shown the theoretical results for Richardson (damped Jacobi) and Gauss-Seidel. An upper bound of the convergence rate of the hybrid method w.r.t. the correction period $M$ is given, which indicates an optimal $M$ to make the hybrid iteration converge fastest. Several numerical examples including the hybrid Richardson (Gauss-Seidel) iteration for the 1-d (2-d) Poisson equation are presented to verify our theoretical results. In addition, we tested the hybrid iteration utilizing the multigrid method, it still reflects an excellent acceleration effect, unless the level of multigrid is very high. Furthermore, we have achieved the hybrid iteration for the Poisson equation with inhomogeneous boundary condition, which has not been studied in the work of HINTS. So far, we are able to deal with the Poisson equation in which all the parameters are changing. As a meshless acceleration method, it is provided with enormous potentials for practice applications.

\bibliographystyle{abbrv}
\bibliography{references}

\end{document}